\documentclass{article}
\let\oldvec\vec
\let\vec\oldvec

\usepackage{lmodern}
\usepackage[utf8]{inputenc}
\usepackage[T1]{fontenc}

\title{An Abstract Model for Branching and its Application to Mixed Integer Programming}

\makeatletter
\@ifclassloaded{svjour3}
  {
    \author{Pierre Le Bodic\protect\footnote{Now at the Faculty of Information Technology of Monash University.} \and George Nemhauser}
    \institute{Georgia Institute of Technology \email{lebodic@gatech.edu}}
    \titlerunning{An Abstract Model for Branching and its Application to MIP}
    
    \newcommand{\svjourabstractend}{\keywords{ Branch and Bound, Abstract Model, Mixed Integer Programming, Computational Complexity, Algorithm Analysis} \subclass{90C11 \and 90C60 \and 68Q25}}
    
    \usepackage{amsmath}
    \spnewtheorem*{theorem*}{Theorem}{\bf}{\it}
    \spnewtheorem*{corollary*}{Corollary}{\bf}{\it}
    
    \newcommand{\problemnewline}{\\}
    \spnewtheorem*{tmpproblem*}{Problem:}{\bfseries}{\itshape}
    \newenvironment{customproblem}{\begin{theopargself}\begin{tmpproblem*}}{\end{tmpproblem*}\end{theopargself}}
    
   \newcommand{\thankreviewers}{ The authors would like to extend their gratitude to the reviewers for the numerous improvements they contributed to through their comments.}
   
   \newcommand{\announceOS}{A more extensive report of the numerical experiments is provided in the online supplement.}
   \newcommand{\OS}{}

  }
  {
    \author{Pierre Le Bodic\footnote{Now at the Faculty of Information Technology of Monash University.}~ and George Nemhauser \\ Georgia Institute of Technology}

    \newcommand{\svjourabstractend}{}
    
    \newenvironment{acknowledgements}[0]{\section*{Acknowledgments}}{}
    
    \usepackage{amsmath,amsthm}
    \newtheorem{theorem}{Theorem}
    
    \newtheorem{corollary}[theorem]{Corollary}
    \newtheorem{conjecture}{Conjecture}
    
    \newtheorem{proposition}[theorem]{Proposition}
    
    \newtheorem{definition}{Definition}

    \newtheorem*{theorem*}{Theorem}
    \newtheorem*{lemma*}{Lemma}
    \newtheorem*{corollary*}{Corollary}

    \makeatletter
    \def\thmhead@plain#1#2#3{
      \thmname{#1}\thmnumber{\@ifnotempty{#1}{ }\@upn{#2}}
      \thmnote{ {\the\thm@notefont#3}}}
    \let\thmhead\thmhead@plain
    \makeatother
    \newtheoremstyle{customproblem}
      {\topsep}{\topsep}
      {\itshape}{}
      {\bfseries}{}
      {\newline}{}
    \theoremstyle{customproblem}
    \newtheorem*{customproblem}{Problem}
    
    \newcommand{\problemnewline}{}
    
    \newcommand{\thankreviewers}{}
    
    \newcommand{\announceOS}{A more extensive report of the numerical experiments is provided in the appendices, Section \ref{sec:app:all-res}.}
    \newcommand{\OS}{
\section{Detailed results for the experiments on MIP instances}\label{sec:app:all-res}
We present in Section \ref{sec:app:bench-res} and \ref{sec:app:tree-res} the detailed results for the benchmark and tree test sets, as described in Section \ref{sec:mipexpe-bench} and \ref{sec:mipexpe-tree}, respectively.
There are three tables in each section, one for each of the scoring function tested, namely \sfproduct{}, \sfratio{}, and \sfonevar{}.
In each table, each line corresponds to all 10 permutations for each instance of the test set.
The second column indicates the number of permuted instances solved.
We provide statistics in terms of time, number of nodes and LP iterations.
For each of these three measures, we give, from left to right in the table, the minimum, average over solved instances (and average over all 10 instances), and maximum.
In the node and LP iterations columns, we use the letters \emph{k}, \emph{m} and \emph{b} as a shorthand for thousands, millions and billions.
Total and averages are provided at the end of each table.
Note that it would not be fair to compare the measures over solved instances in this setting.
Indeed, the set of solved instances differ depending on the scoring function considered.
Function \sfratio{} generally solves more hard instances, thus the averages on solved instances have higher values.

\subsection{Benchmark test set results}\label{sec:app:bench-res}
Tables \ref{tab:app-miplibs-prod}, \ref{tab:app-miplibs-ratio} and \ref{tab:app-miplibs-onevar} give the results on the benchmark test set for the scoring functions \sfproduct{}, \sfratio{}, and \sfonevar{}, respectively.
\begin{center}
\begin{scriptsize}
\begin{landscape}
{\setlength{\tabcolsep}{5pt}
\begin{longtable}{c|c|ccc|ccc|ccc}
\caption{Detailed results for the \sfproduct{} scoring function on the benchmark test set(continues in the next page).}
\endfoot
\caption{Detailed results for the \sfproduct{} scoring function on the benchmark test set.}
\endlastfoot

     \multirow{2}{*}{Instance} & \multirow{2}{*}{\#solved} & \multicolumn{3}{|c|}{time (s)} & \multicolumn{3}{|c|}{nodes} & \multicolumn{3}{|c}{LP iterations}\\
&& min & avg & max & min & avg & max & min & avg & max\\
\hline
\hline
\endhead
10teams & 10 & 1 & 1 (1) & 1 & 1 & 1 (1) & 1 & 1.0k & 1.2k (1.2k) & 1.6k\\
30n20b8 & 10 & 483 & 521 (521) & 544 & 7 & 10 (10) & 15 & 31.7k & 37.5k (37.5k) & 48.6k\\
a1c1s1 & 0 & 7200 & - (7200) & 7200 & 308.8k & - (433.7k) & 577.7k & 11.9m & - (13.0m) & 14.6m\\
acc-tight5 & 10 & 0 & 0 (0) & 0 & 0 & 0 (0) & 0 & 0 & 0 (0) & 0\\
aflow30a & 10 & 33 & 36 (36) & 42 & 1.8k & 2.5k (2.5k) & 4.2k & 34.5k & 45.1k (45.1k) & 64.9k\\
aflow40b & 10 & 2592 & 3483 (3483) & 4952 & 142.9k & 210.9k (210.9k) & 289.1k & 2.9m & 4.4m (4.4m) & 6.0m\\
air04 & 10 & 28 & 30 (30) & 40 & 6 & 7 (7) & 8 & 11.6k & 13.0k (13.0k) & 14.6k\\
app1-2 & 10 & 603 & 845 (845) & 1362 & 29 & 565 (565) & 2.0k & 8.5k & 28.3k (28.3k) & 84.0k\\
\end{longtable}

\begin{longtable}{c|c|ccc|ccc|ccc}
\caption{Detailed results for the \sfratio{} scoring function on the benchmark test set(continues in the next page).}
\endfoot
\caption{Detailed results for the \sfratio{} scoring function on the benchmark test set.}
\endlastfoot

     \multirow{2}{*}{Instance} & \multirow{2}{*}{\#solved} & \multicolumn{3}{|c|}{time (s)} & \multicolumn{3}{|c|}{nodes} & \multicolumn{3}{|c}{LP iterations}\\
&& min & avg & max & min & avg & max & min & avg & max\\
\hline
\hline
\endhead
10teams & 10 & 0 & 1 (1) & 1 & 1 & 1 (1) & 1 & 1.0k & 1.2k (1.2k) & 1.6k\\
30n20b8 & 10 & 485 & 529 (529) & 571 & 7 & 15 (15) & 34 & 32.6k & 37.1k (37.1k) & 42.9k\\
a1c1s1 & 0 & 7200 & - (7200) & 7200 & 277.0k & - (430.3k) & 545.7k & 10.7m & - (13.4m) & 15.1m\\
acc-tight5 & 10 & 0 & 0 (0) & 0 & 0 & 0 (0) & 0 & 0 & 0 (0) & 0\\
aflow30a & 10 & 34 & 35 (35) & 39 & 1.7k & 2.5k (2.5k) & 4.4k & 34.2k & 43.9k (43.9k) & 61.4k\\
aflow40b & 10 & 1596 & 2372 (2372) & 2974 & 93.8k & 137.2k (137.2k) & 171.5k & 2.0m & 2.9m (2.9m) & 3.6m\\
air04 & 10 & 27 & 29 (29) & 34 & 6 & 7 (7) & 8 & 11.6k & 13.0k (13.0k) & 14.6k\\
app1-2 & 10 & 619 & 758 (758) & 986 & 29 & 572 (572) & 2.0k & 8.4k & 23.6k (23.6k) & 65.9k\\
\end{longtable}

\begin{longtable}{c|c|ccc|ccc|ccc}
\caption{Detailed results for the \sfonevar{} scoring function on the benchmark test set(continues in the next page).}
\endfoot
\caption{Detailed results for the \sfonevar{} scoring function on the benchmark test set.}
\endlastfoot

     \multirow{2}{*}{Instance} & \multirow{2}{*}{\#solved} & \multicolumn{3}{|c|}{time (s)} & \multicolumn{3}{|c|}{nodes} & \multicolumn{3}{|c}{LP iterations}\\
&& min & avg & max & min & avg & max & min & avg & max\\
\hline
\hline
\endhead
10teams & 10 & 1 & 1 (1) & 1 & 1 & 1 (1) & 1 & 1.0k & 1.2k (1.2k) & 1.6k\\
30n20b8 & 10 & 479 & 526 (526) & 605 & 7 & 14 (14) & 36 & 32.6k & 37.5k (37.5k) & 41.7k\\
a1c1s1 & 0 & 7200 & - (7200) & 7200 & 308.8k & - (420.7k) & 510.0k & 11.5m & - (13.8m) & 15.4m\\
acc-tight5 & 10 & 0 & 0 (0) & 0 & 0 & 0 (0) & 0 & 0 & 0 (0) & 0\\
aflow30a & 10 & 31 & 35 (35) & 38 & 1.7k & 2.2k (2.2k) & 3.3k & 32.9k & 41.2k (41.2k) & 52.5k\\
aflow40b & 10 & 1533 & 2437 (2437) & 3328 & 79.0k & 133.3k (133.3k) & 212.9k & 2.0m & 3.2m (3.2m) & 4.9m\\
air04 & 10 & 28 & 31 (31) & 35 & 6 & 7 (7) & 8 & 11.6k & 13.0k (13.0k) & 14.6k\\
app1-2 & 10 & 601 & 754 (754) & 1061 & 29 & 645 (645) & 3.1k & 8.8k & 27.5k (27.5k) & 105.4k\\
\end{longtable}
}
\end{landscape}
\end{scriptsize}
\end{center}

\subsection{Tree test set results}\label{sec:app:tree-res}
Tables \ref{tab:app-tree-prod}, \ref{tab:app-tree-ratio} and \ref{tab:app-tree-onevar} give the results on the tree test set for the scoring functions \sfproduct{}, \sfratio{}, and \sfonevar{}, respectively.
\begin{center}
\begin{scriptsize}
\begin{landscape}
{\setlength{\tabcolsep}{5pt}
\begin{longtable}{c|c|ccc|ccc|ccc}
\caption{Detailed results for the \sfproduct{} scoring function on the tree test set(continues in the next page).}
\endfoot
\caption{Detailed results for the \sfproduct{} scoring function on the tree test set.}
\endlastfoot

     \multirow{2}{*}{Instance} & \multirow{2}{*}{\#solved} & \multicolumn{3}{|c|}{time (s)} & \multicolumn{3}{|c|}{nodes} & \multicolumn{3}{|c}{LP iterations}\\
&& min & avg & max & min & avg & max & min & avg & max\\
\hline
\hline
\endhead
blp-ar98 & 0 & 43200 & - (43200) & 43200 & 1.5m & - (2.0m) & 2.3m & 13.5m & - (17.0m) & 20.2m\\
csched007 & 10 & 26734 & 31000 (31000) & 39186 & 4.8m & 5.7m (5.7m) & 7.0m & 101.7m & 124.5m (124.5m) & 152.4m\\
csched008 & 10 & 750 & 1440 (1440) & 3065 & 58.2k & 112.4k (112.4k) & 221.9k & 7.4m & 15.4m (15.4m) & 33.7m\\
enlight15 & 0 & 43200 & - (43200) & 43200 & 20.1m & - (21.9m) & 22.8m & 101.4m & - (111.2m) & 116.3m\\
enlight16 & 0 & 43200 & - (43200) & 43200 & 23.2m & - (24.1m) & 25.3m & 99.6m & - (103.6m) & 109.1m\\
glass4 & 10 & 15 & 28 (28) & 45 & 5.4k & 11.5k (11.5k) & 18.9k & 36.3k & 75.2k (75.2k) & 130.2k\\
gmu-35-40 & 0 & 43200 & - (43200) & 43200 & 64.2m & - (70.4m) & 78.4m & 157.1m & - (172.4m) & 186.6m\\
gmu-35-50 & 0 & 43200 & - (43200) & 43200 & 51.1m & - (58.6m) & 72.5m & 165.5m & - (213.7m) & 246.4m\\
\end{longtable}

\begin{longtable}{c|c|ccc|ccc|ccc}
\caption{Detailed results for the \sfratio{} scoring function on the tree test set(continues in the next page).}
\endfoot
\caption{Detailed results for the \sfratio{} scoring function on the tree test set.}
\endlastfoot

     \multirow{2}{*}{Instance} & \multirow{2}{*}{\#solved} & \multicolumn{3}{|c|}{time (s)} & \multicolumn{3}{|c|}{nodes} & \multicolumn{3}{|c}{LP iterations}\\
&& min & avg & max & min & avg & max & min & avg & max\\
\hline
\hline
\endhead
blp-ar98 & 0 & 43200 & - (43200) & 43200 & 1.6m & - (1.9m) & 2.3m & 14.9m & - (17.2m) & 19.1m\\
csched007 & 10 & 20695 & 24218 (24218) & 28186 & 3.7m & 4.2m (4.2m) & 4.8m & 90.5m & 101.9m (101.9m) & 123.7m\\
csched008 & 10 & 824 & 1737 (1737) & 2992 & 63.6k & 115.6k (115.6k) & 179.2k & 8.6m & 19.7m (19.7m) & 35.7m\\
enlight15 & 0 & 43200 & - (43200) & 43200 & 25.6m & - (26.9m) & 28.0m & 106.2m & - (111.8m) & 117.5m\\
enlight16 & 0 & 43200 & - (43200) & 43200 & 23.8m & - (25.1m) & 26.7m & 106.6m & - (113.0m) & 120.5m\\
glass4 & 10 & 15 & 38 (38) & 101 & 5.5k & 15.4k (15.4k) & 35.9k & 41.5k & 102.8k (102.8k) & 242.0k\\
gmu-35-40 & 2 & 7147 & 15954 (37751) & 43200 & 11.1m & 23.6m (61.0m) & 75.4m & 26.0m & 54.0m (158.9m) & 217.1m\\
gmu-35-50 & 0 & 43200 & - (43200) & 43200 & 50.0m & - (62.7m) & 73.9m & 169.3m & - (200.3m) & 228.2m\\
\end{longtable}

\begin{longtable}{c|c|ccc|ccc|ccc}
\caption{Detailed results for the \sfonevar{} scoring function on the tree test set(continues in the next page).}
\endfoot
\caption{Detailed results for the \sfonevar{} scoring function on the tree test set.}
\endlastfoot

     \multirow{2}{*}{Instance} & \multirow{2}{*}{\#solved} & \multicolumn{3}{|c|}{time (s)} & \multicolumn{3}{|c|}{nodes} & \multicolumn{3}{|c}{LP iterations}\\
&& min & avg & max & min & avg & max & min & avg & max\\
\hline
\hline
\endhead
blp-ar98 & 0 & 43200 & - (43200) & 43200 & 1.4m & - (1.7m) & 1.9m & 13.9m & - (17.4m) & 20.3m\\
csched007 & 10 & 18507 & 25355 (25355) & 37366 & 3.0m & 4.3m (4.3m) & 5.9m & 77.9m & 105.6m (105.6m) & 143.6m\\
csched008 & 10 & 555 & 1641 (1641) & 4820 & 46.1k & 118.6k (118.6k) & 267.0k & 5.2m & 18.3m (18.3m) & 59.2m\\
enlight15 & 0 & 43200 & - (43200) & 43200 & 23.1m & - (24.8m) & 25.9m & 103.0m & - (111.6m) & 117.4m\\
enlight16 & 0 & 43200 & - (43200) & 43200 & 21.8m & - (25.2m) & 26.3m & 97.5m & - (113.6m) & 119.0m\\
glass4 & 10 & 17 & 29 (29) & 51 & 5.5k & 11.8k (11.8k) & 22.5k & 37.6k & 74.4k (74.4k) & 147.2k\\
gmu-35-40 & 1 & 39750 & 39750 (42855) & 43200 & 59.9m & 59.9m (71.2m) & 79.2m & 132.3m & 132.3m (190.1m) & 224.3m\\
gmu-35-50 & 0 & 43200 & - (43200) & 43200 & 49.0m & - (61.6m) & 75.1m & 178.0m & - (206.2m) & 249.4m\\
\end{longtable}
}
\end{landscape}
\end{scriptsize}
\end{center} }
    
  }
\makeatother
  
\usepackage{appendix}
\usepackage{amssymb}
\usepackage{dsfont}
\usepackage{url}
\usepackage{hyperref}
\usepackage{graphicx}
\usepackage{tikz, pgfplots}
\usepackage{calc}

\usepackage{subcaption}

\usepackage{multirow}
\usepackage{longtable}
\usepackage{pdflscape}

\usepackage{algorithm}
\usepackage{algpseudocode}

\usepackage{verbatim}

\newcommand{\bb}{B\&B}
\newcommand{\BB}{Branch \& Bound}
\newcommand{\Z}{\mathds{Z}}
\newcommand{\Zp}{\Z_{>0}}

\newcommand{\Q}{\mathds{Q}}

\newcommand{\gap}{GVB}

\newcommand{\nvar}{MVB}
\newcommand{\onevar}{SVB}

\newcommand{\Gap}{\textsc{General Variable Branching}}

\newcommand{\Nvar}{\textsc{Multiple Variable Branching}}
\newcommand{\Onevar}{\textsc{Single Variable Branching}}

\newcommand{\countkp}{\textsc{\#Knapsack}}

\newcommand{\sfproduct}{\texttt{product}}
\newcommand{\sfonevar}{\texttt{svts}}
\newcommand{\sfratio}{\texttt{ratio}}
\newcommand{\sflinear}{\texttt{linear}}

\DeclareMathOperator*{\argmin}{arg\,min}

\renewcommand{\times}{\cdot}

   \pgfplotstableread{
100	100	100	100
99	101.0069798808	105	101.0101010101
98	102.028116116	110	102.0408163265
97	103.0637104115	115	103.0927835052
96	104.1140738637	120	104.1666666667
95	105.1795273528	125	105.2631578947
94	106.2604019566	130	106.3829787234
93	107.357039387	135	107.5268817204
92	108.4697924493	140	108.6956521739
91	109.5990255278	145	109.8901098901
90	110.7451150967	150	111.1111111111
89	111.9084502609	155	112.3595505618
88	113.0894333267	160	113.6363636364
87	114.2884804049	165	114.9425287356
86	115.5060220497	170	116.2790697674
85	116.7425039335	175	117.6470588235
84	117.9983875634	180	119.0476190476
83	119.2741510389	185	120.4819277108
82	120.5702898567	190	121.9512195122
81	121.8873177645	195	123.4567901235
80	123.2257676677	200	125
79	124.5861925931	205	126.582278481
78	125.9691667147	210	128.2051282051
77	127.3752864443	215	129.8701298701
76	128.8051715949	220	131.5789473684
75	130.2594666201	225	133.3333333333
74	131.7388419371	230	135.1351351351
73	133.2439953398	235	136.9863013699
72	134.775653509	240	138.8888888889
71	136.3345736287	245	140.8450704225
70	137.9215451163	250	142.8571428571
69	139.5373914785	255	144.9275362319
68	141.1829723012	260	147.0588235294
67	142.8591853881	265	149.2537313433
66	144.5669690593	270	151.5151515152
65	146.3073046257	275	153.8461538462
64	148.0812190553	280	156.25
63	149.8897878491	285	158.7301587302
62	151.7341381478	290	161.2903225806
61	153.6154520902	295	163.9344262295
60	155.5349704501	300	166.6666666667
59	157.4939965782	305	169.4915254237
58	159.4939006809	310	172.4137931034
57	161.5361244722	315	175.4385964912
56	163.622186236	320	178.5714285714
55	165.7536863456	325	181.8181818182
54	167.9323132891	330	185.1851851852
53	170.1598502578	335	188.679245283
52	172.4381823634	340	192.3076923077
51	174.7693045541	345	196.0784313725
50	177.1553303164	350	200
49	179.5985012551	355	204.0816326531
48	182.1011976612	360	208.3333333333
47	184.665950193	365	212.7659574468
46	187.2954528123	370	217.3913043478
45	189.9925771443	375	222.2222222222
44	192.7603884495	380	227.2727272727
43	195.602163433	385	232.5581395349
42	198.5214101496	390	238.0952380952
41	201.521890307	395	243.9024390244
40	204.6076443241	400	250
39	207.7830195616	405	256.4102564103
38	211.0527022205	410	263.1578947368
37	214.4217534947	415	270.2702702703
36	217.8956506785	420	277.7777777778
35	221.4803340647	425	285.7142857143
34	225.1822606442	430	294.1176470588
33	229.0084658214	435	303.0303030303
32	232.9666346309	440	312.5
31	237.0651842615	445	322.5806451613
30	241.3133601173	450	333.3333333333
29	245.7213481677	455	344.8275862069
28	250.3004070196	460	357.1428571429
27	255.0630240178	465	370.3703703704
26	260.0231008165	470	384.6153846154
25	265.1961753641	475	400
24	270.5996892241	480	416.6666666667
23	276.2533118212	485	434.7826086957
22	282.1793368047	490	454.5454545455
21	288.4031706806	495	476.1904761905
20	294.9539407489	500	500
19	301.8652591045	505	526.3157894737
18	309.1761933941	510	555.5555555556
17	316.9325153296	515	588.2352941176
16	325.1883281494	520	625
15	334.0082200159	525	666.6666666667
14	343.470161488	530	714.2857142857
13	353.6694786461	535	769.2307692308
12	364.7244197665	540	833.3333333333
11	376.7841499307	545	909.0909090909
10	390.0405668069	550	1000
9	404.746363169	555	1111.1111111111
8	421.2437728208	560	1250
7	440.0126125889	565	1428.5714285714
6	461.755609142	570	1666.6666666667

}\fixedscoretofive
\pgfplotstableread{
gap	treesize
0	1
1	3
2	3
3	3
4	3
5	3
6	3
7	3
8	3
9	3
10	3
11	7
12	7
13	7
14	7
15	7
16	7
17	7
18	7
19	7
20	7
21	15
22	15
23	15
24	15
25	15
26	15
27	15
28	15
29	15
30	15
31	31
32	31
33	31
34	31
35	31
36	31
37	31
38	31
39	31
40	31
41	63
42	63
43	63
44	63
45	63
46	63
47	63
48	63
49	63
50	63

}\tententofifty
\pgfplotstableread{
gap	treesize
0	1
1	3
2	3
3	3
4	3
5	3
6	3
7	3
8	3
9	3
10	3
11	7
12	7
13	7
14	7
15	7
16	7
17	7
18	7
19	7
20	7
21	15
22	15
23	15
24	15
25	15
26	15
27	15
28	15
29	15
30	15
31	31
32	31
33	31
34	31
35	31
36	31
37	31
38	31
39	31
40	31
41	63
42	63
43	63
44	63
45	63
46	63
47	63
48	63
49	63
50	63
51	127
52	127
53	127
54	127
55	127
56	127
57	127
58	127
59	127
60	127
61	255
62	255
63	255
64	255
65	255
66	255
67	255
68	255
69	255
70	255
71	511
72	511
73	511
74	511
75	511
76	511
77	511
78	511
79	511
80	511
81	1023
82	1023
83	1023
84	1023
85	1023
86	1023
87	1023
88	1023
89	1023
90	1023
91	2047
92	2047
93	2047
94	2047
95	2047
96	2047
97	2047
98	2047
99	2047
100	2047

}\tententohundred
\pgfplotstableread{
10 0
11 22
12 12
13 12
14 12
15 12
16 12
17 12
18 12
19 673
20 574
21 84
22 66
23 46
24 46
25 536
26 536
27 928
28 988
29 1859
30 80
31 72
32 52
33 52
34 44
35 44
36 30
37 30
38 30
39 30
40 30
41 30
42 30
43 30
44 30
45 30
46 30
47 30
48 30
49 30
50 30

}\tentenvstwosomethingHvalues
\pgfplotstableread{
1 3
2 3
3 3
4 3
5 3
6 3
7 3
8 3
9 3
10 3
11 5
12 5
13 7
14 7
15 7
16 7
17 7
18 7
19 7
20 7
21 9
22 9
23 11
24 11
25 13
26 13
27 15
28 15
29 15
30 15
31 17
32 17
33 19
34 19
35 21
36 21
37 23
38 23
39 25
40 25
41 27
42 27
43 29
44 29
45 31
46 31
47 33
48 33
49 35
50 37

}\twofortyninetententofifty
\pgfplotstableread{
1 3
2 3
3 3
4 3
5 3
6 3
7 3
8 3
9 3
10 3
11 5
12 5
13 7
14 7
15 7
16 7
17 7
18 7
19 7
20 7
21 9
22 9
23 11
24 11
25 13
26 13
27 15
28 15
29 15
30 15
31 17
32 17
33 19
34 19
35 21
36 21
37 23
38 23
39 25
40 25
41 27
42 27
43 29
44 29
45 31
46 31
47 33
48 33
49 35
50 37
51 39
52 41
53 43
54 45
55 47
56 49
57 51
58 53
59 55
60 59
61 61
62 67
63 69
64 75
65 77
66 83
67 85
68 91
69 93
70 101
71 103
72 113
73 115
74 127
75 129
76 143
77 145
78 159
79 161
80 177
81 179
82 197
83 199
84 219
85 221
86 243
87 245
88 269
89 271
90 297
91 299
92 327
93 329
94 359
95 361
96 393
97 395
98 429
99 433
100 469

}\twofortyninetententohundred
\pgfplotstableread{
1 3
2 3
3 5
4 5
5 7
6 7
7 9
8 9
9 11
10 11
11 13
12 13
13 15
14 15
15 17
16 17
17 19
18 19
19 21
20 21
21 23
22 23
23 25
24 25
25 27
26 27
27 29
28 29
29 31
30 31
31 33
32 33
33 35
34 35
35 37
36 37
37 39
38 39
39 41
40 41
41 43
42 43
43 45
44 45
45 47
46 47
47 49
48 49
49 51
50 53

}\twofortyninetofifty
\pgfplotstableread{
1 3
2 3
3 5
4 5
5 7
6 7
7 9
8 9
9 11
10 11
11 13
12 13
13 15
14 15
15 17
16 17
17 19
18 19
19 21
20 21
21 23
22 23
23 25
24 25
25 27
26 27
27 29
28 29
29 31
30 31
31 33
32 33
33 35
34 35
35 37
36 37
37 39
38 39
39 41
40 41
41 43
42 43
43 45
44 45
45 47
46 47
47 49
48 49
49 51
50 53
51 55
52 59
53 61
54 67
55 69
56 77
57 79
58 89
59 91
60 103
61 105
62 119
63 121
64 137
65 139
66 157
67 159
68 179
69 181
70 203
71 205
72 229
73 231
74 257
75 259
76 287
77 289
78 319
79 321
80 353
81 355
82 389
83 391
84 427
85 429
86 467
87 469
88 509
89 511
90 553
91 555
92 599
93 601
94 647
95 649
96 697
97 699
98 749
99 753
100 805
}\twofortyninetohundred

\begin{document}
\maketitle
\begin{abstract}
 The selection of branching variables is a key component of branch-and-bound algorithms for solving Mixed-Integer Programming (MIP) problems since the quality of the selection procedure is likely to have a significant effect on the size of the enumeration tree.
 State-of-the-art procedures base the selection of variables on their ``LP gains'', which is the dual bound improvement obtained after branching on a variable.
 There are various ways of selecting variables depending on their LP gains.
 However, all methods are evaluated empirically.
 In this paper we present a theoretical model for the selection of branching variables.
 It is based upon an abstraction of MIPs to a simpler setting in which it is possible to analytically evaluate the dual bound improvement of choosing a given variable.
 We then discuss how the analytical results can be used to choose branching variables for MIPs, and we give experimental results that demonstrate the effectiveness of the method on MIPLIB 2010 ``tree'' instances where we achieve a $5\%$ geometric average time and node improvement over the default rule of SCIP, a state-of-the-art MIP solver.
\svjourabstractend{}
\end{abstract}

\section{Introduction}
 \BB{} (\bb{}) \cite{land60a} is currently the most successful and widely used algorithm to solve general Mixed Integer Programs (MIPs). 
\bb{} searches the solution space by recursively splitting it, which is traditionally represented by a tree, where the root node is associated with the entire solution space and where sibling nodes represent a partition of the solution space of their parent node.
At each node, the subspace is encoded by a MIP, and its Linear Program (LP) relaxation is solved to provide a dual bound.
If the LP is infeasible, or if the LP bound is no better than the primal bound, the node is pruned (the primal bound is the value of the best feasible solution found so far).
Otherwise, the subspace at that node is partitioned and the two corresponding children nodes are recursively explored.
The part of the \bb{} algorithm that decides how to partition the solution space is referred to as the \emph{branching rule}.
Typically, the branching rule selects one variable among the \emph{candidate variables}, i.e. those that have a fractional value in the LP solution of the current node, but are required to be integer in the original MIP.
Formally, suppose that at the current node the value of an integer variable $x$ in the node's LP solution is $x_{lp} \not\in \Z$.
Branching on $x$ would result in two children, each encoding a solution subspace, one in which $x$ is upper-bounded by $\lfloor x_{lp} \rfloor$, the other in which $x$ is lower-bounded by $\lceil x_{lp} \rceil$.
The PhD thesis \cite{achterberg07a} provides a practical point of view of the latest advances in MIP solving and branching in particular.
For a general overview of MIP solving and \bb{}, see \cite{nemhauser88a}, \cite{achterberg13b} and \cite{conforti14a}.

This research is motivated by the desire to understand the fundamentals of state-of-the-art MIP branching rules (many of which are justified experimentally).
The main contribution of this paper is the introduction of theoretical decision problems to study the branching component of \bb{}.
Based on an analysis of these models, we introduce new \emph{scoring functions}, and prove their efficiency on both simulated experiments and MIP instances.
The paper is organized as follows.
After defining the abstract \bb{} model in Section \ref{sec:defs}, we study the simplest of our problems in Section \ref{sec:onevar}.
More complex problems are introduced and analyzed in Sections \ref{sec:nvar} and \ref{sec:gap}, respectively.
Section \ref{sec:scoring} introduces two new scoring functions based on the theory developed in previous sections.
Experimental results are presented in Section \ref{sec:expresults} and conclusions are given in Section \ref{sec:conclusion}.

 \section{Abstract \BB{} model}\label{sec:defs}
In this section we model an abstract version of \bb{} through a series of definitions.
We start by defining what we consider to be a \emph{variable}.
 \begin{definition}
  A variable $x$ is encoded by a pair of non-negative (left and right) integer gains $(l_x, r_x)$ with $1 \leq l_x \leq r_x$.
 \end{definition}
 The $l_x$ and $r_x$ gains model the dual bound changes that occur when branching on the variable $x$.
 We suppose that these two gains are \emph{known} and \emph{fixed}.
 (Note that when solving a MIP using \bb{}, these gains are not fixed for a variable.)
 
 We now define a \emph{\BB{} tree}.
 \begin{definition}
 \label{def:bbtree}
 Given a set of variables, a \BB{} tree is a binary tree (in the graph sense) with one variable $x$ assigned to each inner node $i$
 .
 \end{definition}
 We then say that variable $x$ is \emph{branched on} at node $i$.
 Note that, in Definition \ref{def:bbtree}, the root node is considered as an inner node.
 Also note that, in our abstract model, the gains $(l_x, r_x)$ of a variable $x$ do not change with the depth of the node where the variable is branched on, nor do they depend on the dual bound at this node.
 
 These two definitions naturally lead to the definition of \emph{dual gap closed} at a node $i$ in a \bb{} tree.
 Note that throughout the article, the gap will always refer to the \emph{absolute} gap (as opposed to the relative gap), i.e. the absolute difference between the primal bound and the dual bound.
 \begin{definition}
  The dual gap $g_i$ closed at a node $i$, referred to as $g_i$, is given by the recursive formula
  \begin{equation*}
   g_i = \begin{cases} 0 &\text{ if $i$ is the root node} \\ g_h + l_{x(h)} &\text{ if $i$ is the left child of node $h$} \\ g_h + r_{x(h)} &\text{ if $i$ is the right child of node $h$,}\end{cases}
  \end{equation*}
  where $x(h)$ is the variable branched on at node $h$.
 \end{definition}
 For instance, if at a node $i$, the dual bound is $g_i=10$, and variable $(2, 5)$ is branched on, the dual bound closed at the left and right children of $i$ are $12$ and $15$, respectively.
 
 The gap \emph{closed} by a \bb{} tree is defined as follows:
  \begin{definition}
  A tree closes a gap $G$ if for all leaves $i$, $g_i \geq G$ holds.
 \end{definition}
 Throughout, the \emph{size} of a tree refers to the number of nodes of the tree.
 Using these definitions, the following sections define problems that model the \bb{} algorithm in increasing complexity.

\section{The \Onevar{} problem}\label{sec:onevar} 
  The first problem we define using the abstract model set up in Section \ref{sec:defs} is the \Onevar{} (\onevar{}) problem, which is a tree-size measurement problem.
 \begin{customproblem}[\Onevar{}]\problemnewline{}
 \textbf{Input:}  one variable encoded by $(l, r) \in \Zp^2$, $G \in \Zp$, $k \in \Zp$, such that $l \leq r \leq G$.\\
 \textbf{Question:} Is the size of the \BB{} tree that closes the gap $G$, repeatedly using variable $(l, r)$, at most $k$? 
 \end{customproblem}
  In the example given in Figure \ref{fig:treeex-onevar}, a \bb{} tree that closes a gap of $G=6$ is built using only variables with gains $(2,5)$.
  Observe that any tree closing a gap $G$ contains the \emph{unique} tree that closes $G$ in a minimum number of nodes.
  As we will see, the analysis of \onevar{} will prove more complex and insightful than it may seem at first glance.
 \begin{figure}
 \centering
 \begin{tikzpicture}[scale=0.5,level distance=1.5cm,
level 1/.style={sibling distance=4.5cm},
level 2/.style={sibling distance=2.5cm}]
\tikzstyle{every node}=[minimum size=.75cm, circle,draw]
 \node (a){$0$}
  child {node (b1) {$2$}
    child {node (c1) {$4$}
      child {node (d1) {$6$} }
      child {node (d2) {$9$} }
    }
    child {node (c2) {$7$} }
  }
  child {node (b2) {$5$}
    child {node (c3) {$7$}}
    child {node (c4) {$10$}}
  };
 \end{tikzpicture}
 \caption{\bb{} tree with 9 nodes that closes a gap of 6 with variable gains $(2,5)$. The gap closed at each node is indicated by the value at the center of each node.}
 \label{fig:treeex-onevar}
 \end{figure}
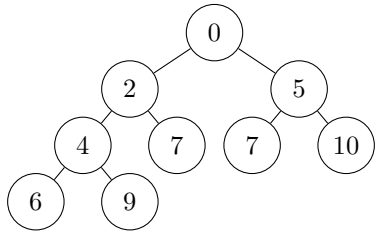

 \subsection{Motivation: state-of-the-art scoring functions and example}\label{sec:motiv-onevar}
A state-of-the-art branching rule for MIP solving (as implemented in SCIP \cite{achterberg09b}) is a hybrid of mainly two methods: \emph{strong branching} and \emph{pseudo-cost branching}.
Given a candidate variable $x$, strong branching computes the LP gains of the would-be children of the current node if $x$ is branched on.
The candidate variable which provide the best LP bounds is then branched on.
Pseudo-cost branching complements strong branching; it keeps track of the LP gains of branching variables for which the children have already been processed, and, given sufficient historical data, estimates the gains that would be computed by strong branching.
Strong branching is computationally expensive, as it requires solving many LPs, while pseudo-cost branching only requires a few arithmetic operations at each node.
However, pseudo-cost branching is an estimate, and it requires initialization.
The state-of-the-art branching rule essentially consists of using strong branching early in the tree, i.e. at the root node and a few subsequent levels, while pseudo-cost branching is called at lower levels.
In practice, there are many refinements to these methods.
We direct the reader to Tobias Achterberg's PhD thesis \cite{achterberg07a} and references therein for an in-depth review of branching rules.

At each node, strong branching (or pseudo-cost branching) provides the (estimated) LP gains $(l_x, r_x)$ resulting from branching up and down for each fractional variable $x$.
Each candidate variable $x$ is then scored according to its gains $(l_x, r_x)$, and the the highest scoring variable is selected for branching.
The state-of-the-art \emph{scoring function} used for this purpose is:
\begin{equation*}
 \max(\epsilon, l_x) \times \max(\epsilon, r_x),
\end{equation*}
where $\epsilon>0$ is chosen close to 0 (e.g. $\epsilon=10^{-6}$ in \cite{achterberg07a}) to break the ties if all variables satisfy $\min(l_x, r_x)=0$, but plays no role otherwise.
We simply refer to this scoring function as the \emph{\sfproduct{}} function.
Prior to the introduction of the \sfproduct{} function, the \emph{\sflinear{}} (or convex) function was the standard:
\begin{equation*}
(1-\mu) \times l_x + \mu \times r_x,
\end{equation*}
where $\mu$ is a parameter in $[0,1]$ ($\mu=\frac{1}{6}$ in \cite{achterberg07a}).
In \cite{achterberg07a}, Achterberg reports that switching from the \sflinear{} function to the \sfproduct{} function yields an improvement of $14\%$ in computing time, and $34\%$ in number of nodes.
The rationale behind the \sflinear{} function with $\mu=\frac{1}{6}$ is that it should be preferable to improve the dual bound of both children by a little rather than only one of them.
The \sfproduct{} function prefers equal left and right gains, i.e. it is more likely to produce a balanced tree.
However, there are types of instances for which the \sflinear{} function performs better, so the \sfproduct{} function does not systematically outperform the \sflinear{} one.

To the best of our knowledge, both the \sflinear{} and the \sfproduct{} function have been established experimentally, and no theoretical evidence supports the use of one over the other, or over any other possible function.
We now give a simple example suggesting that more complex functions are required to score variables.
If the \sfproduct{} or the \sflinear{} functions are good scoring functions for solving MIPs, then they should also perform well on simple models such as \onevar{}.
Consider variables $(10, 10)$ and $(2, 50)$: both the \sfproduct{} and the \sflinear{} function (with default \sflinear{} coefficient $\mu = \frac{1}{6}$) assign these variables an equal score.
Observe that variable $(2, 49)$ receives a smaller score than $(2, 50)$ (and therefore than variable $(10, 10)$).
We now consider the tree-sizes of two \onevar{} instances, one that has variable $(2, 49)$ as input, the other having variable $(10, 10)$.

  \begin{figure}
  \centering
  \begin{subfigure}{.45\columnwidth}
    \begin{tikzpicture}[scale=.7]
    \begin{axis}[xlabel={gap},ylabel={tree-size},]
    \addplot[blue] table {\tententofifty};
    \addplot[red] table {\twofortyninetofifty};
    \end{axis}
  
    \end{tikzpicture}
    \caption{Tree-sizes for gaps in [1, 50]}
  \end{subfigure}
  \begin{subfigure}{.1\columnwidth}
  \end{subfigure}
  \begin{subfigure}{.45\columnwidth}
    \begin{tikzpicture}[scale=.7]
    \begin{axis}[xlabel={gap}]
    \addplot[blue] table {\tententohundred};
    \addplot[red] table {\twofortyninetohundred};
    \end{axis}
    \end{tikzpicture}
    \caption{Tree-sizes for gaps in [1, 100]}
  \end{subfigure}
  \caption{Plot of the sizes of the trees built with variables $(10,10)$ (in blue) and $(2, 49)$ (in red).}
  \label{fig:treesize-ex-onevar}
  \end{figure}
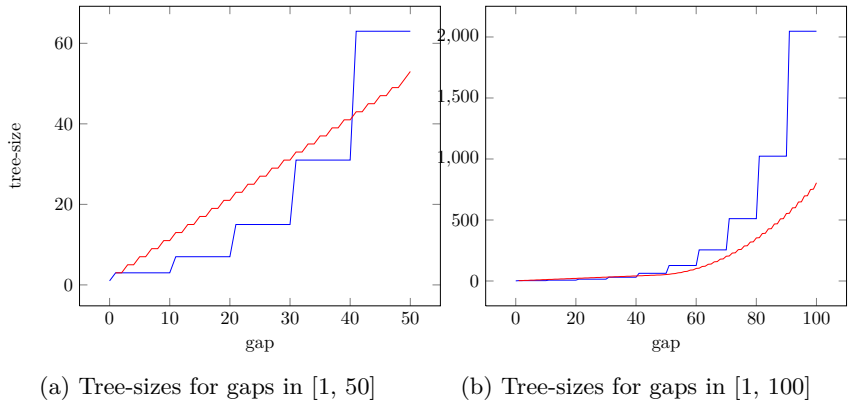

Figure \ref{fig:treesize-ex-onevar} gives two \onevar{} tree-size plots.
For gaps no larger than $40$, variable $(10, 10)$ requires fewer nodes than $(2, 49)$ to close the gap, but larger gaps are closed faster with $(2, 49)$.
For a gap of 1000, using only $(10, 10)$ produces a tree $323$ \emph{million} times larger than using only $(2, 49)$, yet both scoring functions assign a higher score to variable $(10, 10)$.
It is clear from this example that the state-of-the-art scoring functions are imperfect.
 
 \subsection{A polynomial-time algorithm}
There exists a simple recurrence relation that solves \onevar{}.
For a given gap $G$, $t(G)$ is defined as the size of the \onevar{} tree that closes $G$:
 \begin{equation}
   t(G) = \begin{cases}
      1 &\text{ if } G \leq 0\\
      1 + t(G-l) + t(G-r) &\text{ otherwise.}         
         \end{cases}
         \label{eq:onevar-t}
 \end{equation}
 
 Hence, for a given $k$, the answer to SVB is YES if and only if $t(G) \leq k$.
 Unfortunately, computing $t(G)$ requires $O(G)$ operations.
 Therefore the running time of this algorithm is pseudo-polynomial (and thus exponential) in the encoding length of the input.
 Moreover, observe that we do not necessarily have $G=O(\log(k))$ for YES instances of SVB: if the given variables are all $(1, G)$, then $k=2G +1$ is the tree-size.
 However, there exists a closed-form formula that can be evaluated in polynomial time:  
 \begin{theorem}\label{th:cff}
  \onevar{} is in P.
  Furthermore, a formula for the size of the \bb{} tree is
  \begin{equation}
   t(G) = 1 + 2 \times \sum_{k_r=1}^{\lceil\frac{G}{r}\rceil} \binom{ k_r + \lceil \frac{G-(k_r-1) \times r}{l} \rceil- 1 }{k_r}.
   \label{eq:onevar-cff}
  \end{equation}
 \end{theorem}
  \begin{proof}
 The size of the input is $O(\log(G)+\log(k))$.
 Let $T$ be the minimal tree that closes the gap $G$, and let $d = \lceil\frac{G}{r}\rceil$ be the \emph{minimum} depth of its leaf nodes.
 If $k< 2^{d+1}-1$, the answer to the decision problem \onevar{} is necessarily NO, therefore we can assume the opposite throughout the rest of proof.
 It follows that $d = O(log(k))$.
 Starting from the root of $T$, we turn right at most $d$ times before finding a leaf.
 For a given path from the root to a leaf, let $k_l$ (resp. $k_r$) denote the number of times we turn left (resp. right).
 Observe that a pair $(k_l, k_r)$ does not uniquely identify a leaf.
 Furthermore, since the number of leaves is in $O(k)$, it is impossible to iterate over all of them in a time polynomial in the size of the input.
 However, since all leaves satisfy $k_r \leq d$, it is possible to iterate over all the values of $k_r$. 
 From now on we suppose that $k_r \leq d$ is fixed, and restrict our attention to those leaves that are reached by turning right exactly $k_r$ times.
 The inequality
 \begin{equation*}
   k_l \times l + k_r \times r < G + r
 \end{equation*}
 must hold for each leaf, since the gap closed at each leaf cannot exceed $G$ by $r$ or more, as otherwise the gap would already be closed at the leaf's parent.
 We therefore have the bound $k_l \leq \lceil\frac{G-r(k_r-1)}{l}\rceil -1$, and the depth of a leaf is bounded by $k_r+\lceil\frac{G-r(k_r-1)}{l}\rceil -1$.
 Furthermore, observe that any $k_r$ elements chosen in the set $\{1, \dots, k_r+\lceil\frac{G-r(k_r-1)}{l}\rceil -1 \}$ uniquely determine a path to a leaf, and that each leaf can be encoded using this scheme.
 This bijection thus ensures that the number of such leaves is $\binom{k_r+\lceil\frac{G-r(k_r-1)}{l}\rceil -1}{k_r}$.
 The total number of leaves is therefore 
  \begin{equation}
   \sum_{k_r=0}^{\lceil\frac{G}{r}\rceil} \binom{ k_r + \lceil \frac{G-(k_r-1) \times r}{l} \rceil- 1 }{k_r}.
   \label{eq:totalleaves}
  \end{equation}
Formula \eqref{eq:totalleaves} can be computed in $O(d^2)=O(\log^2(k))$ operations, which is polynomial in the size of the input.
Observe how the formula iterates over the possible $k_r$'s but avoids iterating over the range of $k_l$'s, as this may require $O(k)$ steps.
Since each inner node has exactly two children, the theorem is proven.
 \end{proof}
 
 Evaluating formula \eqref{eq:onevar-cff} becomes impractical as the gap $G$ becomes large.
 We present below an asymptotic study of the tree-size that can be used in practice when formula \eqref{eq:onevar-cff} is too expensive to compute.
 
 \subsection{Asymptotic study}\label{sec:onevar-asympt}
 Analyzing algorithms through a related recurrence relation similar to \eqref{eq:onevar-t} (and its mathematical properties, such as the asymptotic behavior) is a standard technique (see \cite[Section 4.5]{cormen09a}).
 In this section and throughout the article we however systematically provide self-sufficient explanations and analyzes, as our readership may not be familiar with algorithm analysis techniques.
 
 As the gap to close $G$ becomes large, we prove that the tree-size grows essentially linearly, i.e. there exists a fixed \emph{ratio} determining the growth of the tree.
 Perhaps the most popular example of this phenomenon is the Fibonacci sequence, for which the golden ratio $\frac{1+\sqrt{5}}{2}$ corresponds to the asymptotic growth rate between two consecutive numbers. 
 In fact, the Fibonacci sequence is given by the recursion formula \eqref{eq:onevar-t} with variable $(1,2)$ up to the additive constant $1$, which does not intervene in asymptotic results.
 This section therefore provides convergence results on sequences generalizing the Fibonacci sequence.
 Unfortunately, even though using the same notion for the ratio as in the Fibonacci sequence would be the most intuitive, the limit of $\frac{t(G+1)}{t(G)}$ when $G$ tends to infinity may not exist in our case.
 Indeed, consider variable $(2, 2)$, and the limit of these two subsequences:
 \begin{align*}
  \lim_{\substack{G \rightarrow \infty \\ G \text{ odd}}} \frac{t(G+1)}{t(G)} = 1 && \lim_{\substack{G \rightarrow \infty \\ G \text{ even}}} \frac{t(G+1)}{t(G)} = 2.
 \end{align*}
 In this example, the ratio $\frac{t(G+1)}{t(G)}$ does not converge as $G$ tends to infinity.
 However, there exists another definition of the ratio, that can be shown to converge.
 We use the following definition:
 \begin{definition}
 \label{def:ratio-onevar}
  The ratio of a variable is
 \begin{equation*}
  \varphi = \lim_{G \rightarrow \infty} \sqrt[l]{\frac{t(G+l)}{t(G)}}.
 \end{equation*} 
 \end{definition}
 With this definition, variable $(2, 2)$ has ratio $\varphi = \sqrt{2}$.
 This means that, asymptotically, the number of nodes doubles for every two additional units of gap to close.
 
 Since the ratio $\varphi$ indicates the growth rate of a tree in the SVB setting, a variable with a small ratio requires a smaller SVB tree than one with a bigger ratio, given a large enough gap $G$.
 The ratios of the two variables taken as example in Section \ref{sec:motiv-onevar}, namely $(10, 10)$ and $(2, 49)$, are approximately 1.071 and 1.049, respectively.
 This explains why variable $(2, 49)$ produces smaller trees than $(10,10)$ for large enough gaps.
 Throughout this section we show that the limit defining the ratio $\varphi$ in Definition \ref{def:ratio-onevar} always exists and how it can be computed.
We first prove Proposition \ref{prop:phi with l or r}, which gives an equivalent way of defining $\varphi$.
\begin{proposition}
 If the limit $\varphi$ exists, then we also have
 \begin{equation*}
  \varphi = \lim_{G \rightarrow \infty} \sqrt[r]{\frac{t(G+r)}{t(G)}}.
 \end{equation*}
 \label{prop:phi with l or r}
\end{proposition}
\begin{proof}
 \begin{align*}
  \varphi^{l r} 
  & = \left(\lim_{G \rightarrow \infty} \frac{t(G+l)}{t(G)}\right)^r\\
  & = \lim_{G \rightarrow \infty} \frac{t(G+l)}{t(G)} \times\frac{t(G+l)}{t(G)} \times \dots \times \frac{t(G+l)}{t(G)}\\
  & = \lim_{G \rightarrow \infty} \frac{t(G+l r)}{t(G + l (r-1))} \times\frac{t(G + l (r-1))}{t(G + l (r-2))} \times \dots \times \frac{t(G + l)}{t(G)}\\
  & = \lim_{G \rightarrow \infty} \frac{t(G+l r)}{t(G)} \\
  & = \lim_{G \rightarrow \infty} \frac{t(G+l r)}{t(G + (l-1) r))} \times\frac{t(G +(l-1) r)}{t(G + (l-2) r)} \times \dots \times \frac{t(G + r)}{t(G)}\\
  & =  \left(\lim_{G \rightarrow \infty} \frac{t(G+r)}{t(G)}\right)^l .
 \end{align*}
\end{proof}

\begin{proposition}
 If the limit $\varphi$ exists, then $\sqrt[r]{2} \leq \varphi \leq \sqrt[l]{2}$.
 \label{prop:phibounds}
\end{proposition}
\begin{proof}
Let $T$ (resp. $T^l$, $T^r$) be the tree that closes the gap $G$ (resp. $G+l$, $G+r$), and observe that the tree $T$ is a sub-tree of $T^l$ and $T^r$ (all three have the same root).
First, each leaf of $T$ is an inner node of $T^r$.
This means that the number of nodes at least asymptotically doubles, thus $2$ is a lower bound on $\varphi^r$.
Second, each leaf of $T$ is either a leaf or the parent of two leaves in $T^l$, hence $2$ is an upper bound on $\varphi^l$.
\end{proof}
The limit $\varphi$ clearly exists in the case $l=r$, and the bounds in the above property provide the exact value.
 We now prove that the limit $\varphi$ exists in general:
 \begin{theorem}
 When $G$ tends to infinity, both sequences $\sqrt[l]{\frac{t(G+l)}{t(G)}}$ and $\sqrt[r]{\frac{t(G+r)}{t(G)}}$ converge to $\varphi$, which is the unique root greater than $1$ of the equation $p(x)= x^r - x^{r-l} -1=0$. 
 \label{th:onevar-cvg}
 \end{theorem}
 \begin{proof}
  See Appendix \ref{app-sec:proof-onevar} and Theorem \ref{th:pol}.
 \end{proof}

\begin{corollary}
A numerical approximation of $\varphi^r$ is given by the fixed-point iteration
\begin{equation*}
 f(x) = 1  + \frac{1}{x^\frac{l}{r} -1}
\end{equation*}
with starting point $x=2$.
\label{cor:fpeq}
\end{corollary} 
 \begin{proof}
  See Appendix \ref{app-sec:proof-onevar}.
 \end{proof}
  
Having established Theorem \ref{th:onevar-cvg}, the definitions of $\varphi$ and $p$ should now make more sense, as they can be shown to originate more directly from the definition of the sequence:
\begin{align*}
 &t(G) = 1 + t(G-l) + t(G-r) \\
 \Rightarrow& \frac{t(G)}{t(G-r)} = \frac{1}{t(G-r)} + \frac{t(G-l)}{t(G-r)} + 1\\
 \Rightarrow &\lim_{G \rightarrow \infty}\frac{t(G)}{t(G-r)} =\lim_{G \rightarrow \infty} \left( \frac{1}{t(G-r)} + \frac{t(G-l)}{t(G)} \frac{t(G)}{t(G-r)} + 1 \right)\\
 \Rightarrow & \varphi^r = \varphi^{r-l} + 1.
 \end{align*}
 The polynomial $p:x \rightarrow x^r - x^{r-l} -1$ is the so-called \emph{characteristic polynomial} of the recurrence sequence defining $t$.
 The Abel–Ruffini theorem \cite[p. 264]{jacobson09a} states that there is no general closed-form formula for roots of polynomials with degrees five or higher.
 Since we are dealing with a particular trinomial, a formula still might exist, however we have not been able to determine one.
 Throughout the paper we therefore resort to numerical methods to determine $\varphi$ (see Section \ref{sec:sfratio}). 
 The characteristic polynomial for trees with arbitrary degrees has been discussed at length in a SAT setting in \cite{kullmann09a}, where the ratio $\varphi$ is referred to as the $\tau$-value.
 
 Examples of variables and their respective ratios are given in Table \ref{tab:phiexamples}.
 In the case where the sum of the variable gains is fixed (Table \ref{tab:phiexamples-plus}), choosing $l$ close to $r$ minimizes the ratio of a variable $(l, r)$.
 However, if the product is fixed (Table \ref{tab:phiexamples-multiply}), then the opposite is preferable.
 \begin{table}
 \centering
    \begin{subtable}{\linewidth}
 \centering
  \begin{tabular}{c|cccccccc}
   Variable & (6, 10) & (5, 11)& (4, 12) & (3, 13) & (2, 14) & (1, 15)\\
   \hline
   $\varphi$ & 1.0926 & 1.0955 & 1.1002 & 1.107 &  1.1204 & 1.1468\\
  \end{tabular}  
  \caption{Variables that satisfy $l+r=16$}
  \label{tab:phiexamples-plus}
  \end{subtable}
 
  \begin{subtable}{\linewidth}
 \centering
  \begin{tabular}{c|cccccc}
   Variable & (6, 10) & (5, 12) & (4, 15) & (3, 20) & (2, 30) & (1, 60)\\
   \hline
   $\varphi$ & 1.0926 & 1.0907 & 1.0873 & 1.0813 & 1.0709 & 1.0515\\
  \end{tabular}
  \caption{Variables that satisfy $l \times r=60$}
  \label{tab:phiexamples-multiply}
  \end{subtable}
  \caption{Ratios for some variables, truncated to 5 significant digits.}
\label{tab:phiexamples}
 \end{table}
 Figure \ref{fig:ratio-3dplot} gives a plot of the ratio as a function of the left and right gains.
 \begin{figure}
 \IfFileExists{./dat/ratio-3dplot.png}{
    \includegraphics[width=\linewidth]{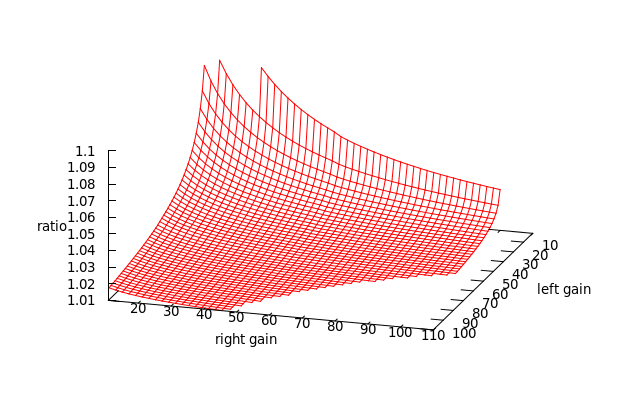}
 }
{
    \includegraphics[width=\linewidth]{ratio-3dplot.png}
}
 \caption{Variable ratio for left and right gains in $[10, 100]$.}
  \label{fig:ratio-3dplot}
 \end{figure}
 
 Once the value of $\varphi$ is determined, it is then easy to circumvent the $O(\log^2(k))$ complexity of the closed-form Formula \eqref{eq:onevar-cff}, by evaluating $t(F)$ for some value $F$ much smaller than $G$, and approximating $t(G)$ using the formula
 \begin{equation*}
  t(G) \approx \varphi^{G-F} t(F).
  \label{eq:onevar-ts-approx}
 \end{equation*}
 We now establish some useful properties of $p$.
 \begin{theorem}
  The characteristic polynomial $p:x \rightarrow x^r - x^{r-l} -1$ satisfies the following properties.
 \begin{itemize}
  \item $p$ is monotonically increasing in $[1, \infty)$
  \item in $[1, \infty)$, $p$ has a single real root $\varphi > 1$
 \end{itemize}
 \label{th:pol}
 \end{theorem}
\begin{proof}
Consider the sign of $p'(x)= r x^{r-1} - (r-l) x^{r-l-1}$ when $x \in R^+$.
\begin{align*}
 p'(x) > 0 &\Leftrightarrow x^{r-l-1} ( r x^l -r +l) >0\\
           &\Leftrightarrow r x^l -r +l > 0\\
           &\Leftrightarrow x > \sqrt[l]{1 - \frac{l}{r}}\\
           &\Leftarrow x \geq 1.
\end{align*}
We thus know that $p$ is increasing over $[1, \infty)$.
Furthermore, $p(1)=-1$ and $\lim_{x\rightarrow \infty} p(x) = \infty$, and $p$ is continuous, therefore there exists a single root $\varphi$ in $(1, \infty)$.
\end{proof}
This analysis is corroborated by the fact that the tree-size increases when incrementing the gap by $l$ or $r$, thus we should have $\varphi > 1$.

\section{The \Nvar{} problem}\label{sec:nvar}
We consider the \Nvar{} problem (\nvar{}), which naturally extends the \onevar{} problem, defined as follows:  
 \begin{customproblem}[\Nvar{}]\problemnewline{}
 \textbf{Input:} $n$ variables encoded by $(l_i, r_i), i \in \{1, \dots,n\}$, an integer $G>0$, an integer $k>0$.\\
 \textbf{Question:} Is there a \BB{} tree with at most $k$ nodes that closes the gap $G$, using each variable as many times as needed?
 \end{customproblem}
 Figure \ref{fig:treeex-nvar} illustrates the difference between \nvar{} and \onevar{}, given two variables $(2, 5)$ and $(3, 3)$, and $G=7$.
 In the \nvar{} tree, variable $(2, 5)$ is branched on at all (inner) nodes except the bottom left one, where variable $(3, 3)$ is branched on.
 
There is a recursive equation for \nvar{}, similar to \eqref{eq:onevar-t} for \onevar{}.
For a given gap $G$, $t(G)$ is defined as the minimum size of a \nvar{} tree that closes $G$:
 \begin{equation}
   t(G) = \begin{cases}
      1 &\text{ if } G \leq 0\\
      1 + \min\limits_{1\leq i \leq n}\left(t(G-l_i) + t(G-r_i)\right) &\text{ otherwise.}         
         \end{cases}
         \label{eq:nvar-t}
 \end{equation}
 The time complexity of this algorithm is $O(nG)$, which is pseudo-polynomial in the size of the input.
 Unlike \onevar{}, we do not know whether \nvar{} can be solved in polynomial time.
 For a given $G$, we say that variable $i$ is \emph{branched on} at the corresponding node if $i \in \argmin_{1\leq i \leq n}(t(G-l_i) + t(G-r_i))$.
 
 \begin{figure}
 \centering
 \begin{subfigure}[b]{.5\columnwidth}
 \centering
\begin{tikzpicture}[scale=0.5,level distance=1.5cm,
level 1/.style={sibling distance=4.5cm},
level 2/.style={sibling distance=2.5cm}]
\tikzstyle{every node}=[minimum size=.75cm, circle,draw]
 \node (a){$0$}
  child {node (b1) {$2$}
    child {node (c1) {$4$}
      child {node (d1) {$6$} 
      child {node (e1) {$8$} }
      child {node (e2) {$11$} }
      }
      child {node (d2) {$9$} }
    }
    child {node (c2) {$7$} }
  }
  child {node (b2) {$5$}
    child {node (c3) {$7$}}
    child {node (c4) {$10$}}
  };
 \end{tikzpicture}
 \caption{\onevar{} tree for variable $(2, 5)$ (size=11).}
 \end{subfigure}
  \begin{subfigure}[b]{.45\columnwidth}
  \centering
 \begin{tikzpicture}[scale=0.5,level distance=1.5cm,
level 1/.style={sibling distance=4.5cm},
level 2/.style={sibling distance=2.5cm}]
\tikzstyle{every node}=[minimum size=.75cm, circle,draw]
 \node (a){$0$}
  child {node (b1) {$2$}
    child {node (c1) {$4$}
      child {node (d1) {$7$} }
      child {node (d2) {$7$} }
    }
    child {node (c2) {$7$} }
  }
  child {node (b2) {$5$}
    child {node (c3) {$7$}}
    child {node (c4) {$10$}}
  };
 \end{tikzpicture}
 \caption{Minimum-size \nvar{} tree (size=9).}
 \end{subfigure}
  \begin{subfigure}[b]{1\columnwidth}
  \centering
 \begin{tikzpicture}[scale=0.5,level distance=1.5cm,
level 1/.style={sibling distance=8cm},
  level 2/.style={sibling distance=4cm},
  level 3/.style={sibling distance=2cm}]
\tikzstyle{every node}=[minimum size=.75cm, circle,draw]
 \node (a){$0$}
  child {node (b1) {$3$}
    child {node (c1) {$6$}
      child {node (d1) {$9$} }
      child {node (d2) {$9$} }
    }
    child {node (c2) {$6$}
      child {node (d3) {$9$} }
      child {node (d4) {$9$} }
    }
  }
  child {node (b2) {$3$}
    child {node (c3) {$6$}
      child {node (d5) {$9$} }
      child {node (d6) {$9$} }
    }
    child {node (c4) {$6$}
      child {node (d7) {$9$} }
      child {node (d8) {$9$} }
    }
  };
 \end{tikzpicture}
  \caption{\onevar{} tree for variable $(3, 3)$ (size=15).}
 \end{subfigure}
 
 \caption{For a gap $G=7$, \onevar{} trees using variable $(2, 5)$ or $(3, 3)$, and minimum-size \nvar{} tree using both.}
 \label{fig:treeex-nvar}
 \end{figure}
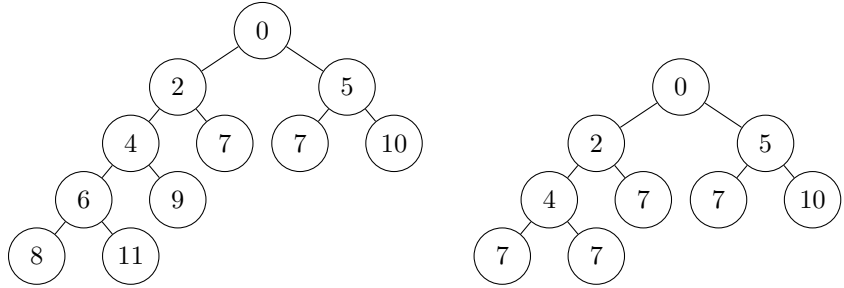
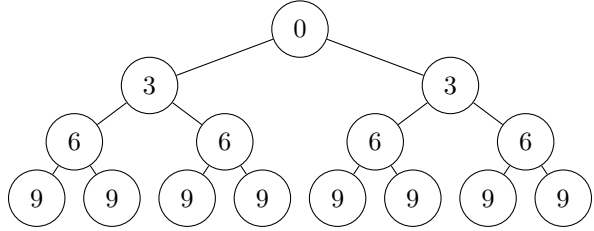
 
 \subsection{Motivating example (continued)}
 Following the example introduced in Section \ref{sec:motiv-onevar} for the \onevar{} problem, consider now the plot of the \nvar{} tree-sizes in Figure \ref{fig:treesize-ex-nvar} (the \onevar{} plots are the same as in Figure \ref{fig:treesize-ex-onevar}).
 
 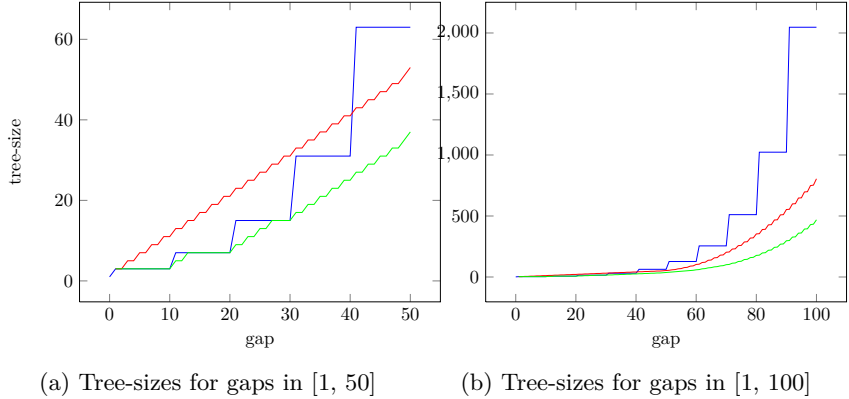
\begin{figure}
\centering
\begin{subfigure}{.45\columnwidth}
  \begin{tikzpicture}[scale=.7]
  \begin{axis}[xlabel={gap},ylabel={tree-size},]
  \addplot[blue] table {\tententofifty};
  \addplot[red] table {\twofortyninetofifty};
  \addplot[green] table {\twofortyninetententofifty};
  \end{axis}
 
  \end{tikzpicture}
  \caption{Tree-sizes for gaps in [1, 50]}
\end{subfigure}
\begin{subfigure}{.1\columnwidth}
\end{subfigure}
\begin{subfigure}{.45\columnwidth}
  \begin{tikzpicture}[scale=.7]
  \begin{axis}[xlabel={gap}]
  \addplot[blue] table {\tententohundred};
  \addplot[red] table {\twofortyninetohundred};
  \addplot[green] table {\twofortyninetententohundred};
  \end{axis}
  \end{tikzpicture}
  \caption{Tree-sizes for gaps in [1, 100]}
\end{subfigure}
\caption{Plot of the sizes of the trees built with variables $(10,10)$ (in blue) and $(2, 49)$ (in red), and both (in green).}
\label{fig:treesize-ex-nvar}
\end{figure}

As anticipated, for a given gap, the \nvar{} tree-size is smaller than each of the \onevar{} tree-sizes.
However, the difference in tree-sizes with $(2, 49)$ does not seem to increase drastically.
Indeed, for a gap of 1000, the \onevar{} tree-size of $(2, 49)$ is only $1.798$ times larger than the \nvar{} tree-size (and this ratio is approximately constant for larger gaps).
In fact, for this instance, we can verify experimentally that variable $(2, 49)$ is branched on at every node for which the gap to close is at least 31.
This phenomenon is the subject of the next section.
 
 \subsection{Asymptotic study}
 Let $\varphi$ (resp. $\varphi_i$) now denote the ratio associated with the recursive equation \eqref{eq:nvar-t} of \nvar{} (resp., for \onevar{}, equation \eqref{eq:onevar-t} for variable $i$).
 Furthermore, for a variable $i$, let $p_i$ be its characteristic polynomial, for $i=1, \dots, n$.
 Formally, we define the ratio for a \nvar{} problem to be
  \begin{equation*}
  \varphi = \lim_{G \rightarrow \infty} \sqrt[z]{\frac{t(G+z)}{t(G)}}
 \end{equation*} 
 where $z$ is the least common multiple of all $l_i$ and $r_i$.
 \begin{theorem}\label{th:nvar-ratio}
  $\varphi = \min_i \varphi_i$
 \end{theorem}
 \begin{proof}
  See Appendix \ref{app-sec:proof-nvar}.
 \end{proof}
Essentially, this means that for a large $G$, the tree-size grows at the growth rate of the \emph{best} variable for \onevar{}.
Therefore we also have 
 \begin{equation*}
  t(G) \approx \varphi^{G-F} t(F)
 \end{equation*}
 for large gaps $F$ and $G$, with $G \geq F$.
Furthermore, we propose the following conjecture.
\begin{conjecture}
For each instance of \nvar{}, there exist a gap $H$ such that for all gaps greater than $H$, variable $i = \argmin_j \varphi_j$ is always branched on at the root node.
\end{conjecture}
Proving this conjecture (and proving that $H$ is polynomially bounded by the size of the input) may help design a polynomial-time algorithm for \nvar{}.
Figure \ref{fig:gapH} provides the value of $H$ for instances of \nvar{} with two variables, $(10, 10)$ and $(2, r)$, where $r$ is the value on the horizontal axis.
The value of $H$ is maximized for $r=29$, which also minimizes $\vert \varphi_{(10, 10)} - \varphi_{(2, r)} \vert$.

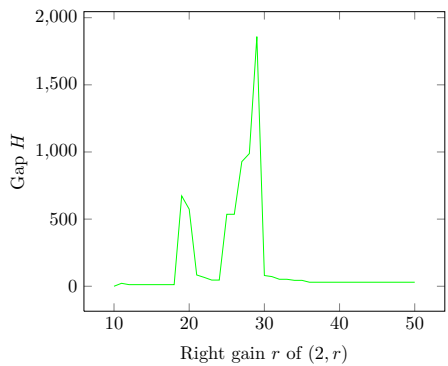
\begin{figure}
\centering
  \begin{tikzpicture}[scale=.7]
  \begin{axis}[xlabel={Right gain $r$ of $(2, r)$}, ylabel={Gap $H$},]
  \addplot[green] table[y index = 1] {\tentenvstwosomethingHvalues};
  \end{axis}
  \end{tikzpicture}
  \caption{Value of the gap $H$ for the \nvar{} problem with variable $(10, 10)$ and $(2, r)$, for varying values of $r$.}
  \label{fig:gapH}
\end{figure}

\section{The \Gap{} problem}\label{sec:gap}
The \Gap{} problem (\gap{}) is defined as follows:
 \begin{customproblem} [\Gap{}]\problemnewline{}
 \textbf{Input:} $n$ variables encoded by $(l_i, r_i), i=1, \dots,n$, an integer $G>0$, an integer $k>0$, and a vector of multiplicities $m \in \Zp^n$.\\
 \textbf{Question:} Is there a \BB{} tree with at most $k$ nodes that closes the gap $G$, branching on each variable $i$ at most $m_i$ times on each path from the root to a leaf?
 \end{customproblem}
 Observe that \gap{} comprises the input of problem \nvar{}, and additionally stipulates that each variable may not be used more than a given number of times on each path from the root to a leaf.
 Consequently, \gap{} corresponds to \nvar{} whenever the multiplicity of each input variable is large enough, therefore \gap{} generalizes \nvar{}.

 When solving a MIP instance using a \bb{} search, the minimization version of \gap{} models the problem of choosing a variable to branch on under the following simplifying hypotheses:
 \begin{itemize}
  \item The up and down gains are known for each variable, and are invariant.
  \item If the problem is feasible, the optimal value is known.
 \end{itemize}
The first hypothesis is not satisfied in practice.
However, as the \bb{} tree grows, pseudo-costs become better approximations for the up and down gains.
Likewise, the second hypothesis becomes true at some point during the \bb{} search.
So, while both hypotheses are not initially satisfied, they become more accurate as the tree grows.
As we will see, these hypotheses are not sufficient to render the \gap{} problem theoretically tractable.
 A recursion for \gap{} similar to \eqref{eq:nvar-t} for \nvar{} is
 \begin{equation}
   t(G, m) = \begin{cases}
      1 &\text{ if } G \leq 0\\
      1 + \min\limits_{1\leq i \leq n, m_i > 0}(t(G-l_i, m-v_i) + t(G-r_i,  m-v_i)) &\text{ otherwise}         
         \end{cases}
   \label{eq:gap-t}
 \end{equation}
 where $m$ and $v_i$, are both vectors of size $n$; $m$ is the vector of multiplicities, and $v_i$ is the indicator vector of the set $\{i\}$ (i.e. the $i^\text{th}$ element of $v_i$ is 1, the $n-1$ others are 0).
 This algorithm keeps track of the variables already branched on, and, in the case where all multiplicities are $1$, its running time is $O(2^n G)$.
 
 Figure \ref{fig:treeex-gap} gives the minimum-size \bb{} tree for a \gap{} instance and an \nvar{} instance that has the same variables and gap as input.
 \begin{figure}
 \centering

  \begin{subfigure}{.45\columnwidth}
  \centering
 \begin{tikzpicture}[scale=0.5,level distance=1.5cm,
level 1/.style={sibling distance=4.5cm},
level 2/.style={sibling distance=2.5cm}]
\tikzstyle{every node}=[minimum size=.75cm, circle,draw]
 \node (a){$0$}
  child {node (b1) {$3$}
    child {node (c1) {$6$} }
    child {node (c2) {$6$} }
  }
  child {node (b2) {$3$}
    child {node (c3) {$6$}}
    child {node (c4) {$6$}}
  };
 \end{tikzpicture}
 \caption{Minimum-size \nvar{} tree (size=7).}
 \end{subfigure}
   \begin{subfigure}{.45\columnwidth}
  \centering
 \begin{tikzpicture}[scale=0.5,level distance=1.5cm,
level 1/.style={sibling distance=4.5cm},
level 2/.style={sibling distance=2.5cm}]
\tikzstyle{every node}=[minimum size=.75cm, circle,draw]
 \node (a){$0$}
  child {node (b1) {$3$}
    child {node (c1) {$5$}
      child {node (d1) {$6$} }
      child {node (d2) {$6$} }
    }
    child {node (c2) {$8$} }
  }
  child {node (b2) {$3$}
    child {node (c3) {$5$}
      child {node (d3) {$6$} }
      child {node (d4) {$6$} }
    }
    child {node (c4) {$8$} }
  };
 \end{tikzpicture}
 \caption{Minimum-size \gap{} tree (size=11).}
 \end{subfigure}
 \caption{Minimum-size trees for \nvar{} and \gap{} with input $G=6$ and variables $(1, 1)$, $(2, 5)$, $(3, 3)$ (in \gap{}, with multiplicities 1).}
 \label{fig:treeex-gap}
 \end{figure}
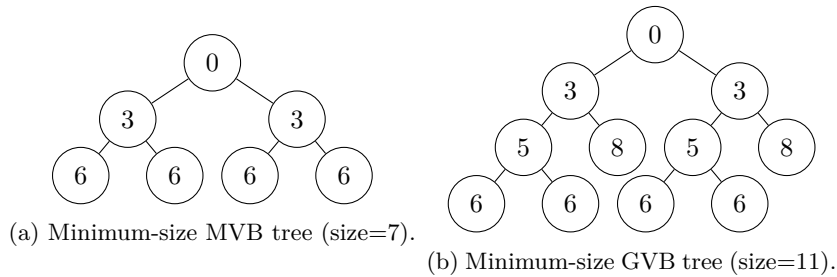

\subsection{Complexity of \gap{}}\label{sec:gap-sharp-p-hard}
For the complexity result presented in this section, it is sufficient to consider the case where all multiplicities are one.
We prove the \#P-hardness of \gap{} by reducing it to a variant of counting knapsack solutions:
 \begin{customproblem} [\countkp{}]\problemnewline{}
\textbf{Input:}  $N$ items with weights $w_1, \dots, w_N$ and $W$ the capacity of the knapsack, a given $K$, all integer positive.\\
\textbf{Question:} Are there at least $K$ distinct feasible solutions $x$ to the (covering) knapsack constraint $\sum_{i=1}^N w_i x_i \geq W$, with $x_i \in \{0,1\}, i=1,\dots, N$?
 \end{customproblem}
 Note that in our definition of the \countkp{} problem, we ask for \emph{at least} $K$ solutions.
This variant of the counting problem can solve the exact counting problem by e.g. dichotomy in $O(N)$ calls, hence it is \#P-hard.
\begin{theorem}
\label{th:sharp-p-hard}
 \gap{} is (weakly) \#P-hard
\end{theorem}
\begin{proof}
We build an instance of the \gap{} problem that embeds an instance of the \countkp{} problem.
Set $n=N+1$, and for each $i\in\{1,\dots,n-1\}$, create a variable with gains $(C, C + w_i)$, where $C=\sum_{j\in\{1,\dots,N\}} w_j$; furthermore, create the $n^{th}$ variable $(C, C)$.
Set the target gap $G = (n-1)C + W$ and the number of allowed nodes $k = 2^{n + 1} -1 -2K$.
Finally, set all multiplicities to one.
Suppose, without loss of generality, that the knapsack instance is feasible, and thus $C\geq W$.
We prove that the answer to this \gap{} instance is YES if and only if the answer to the instance of \countkp{} is YES.

First, observe that a minimum-size tree is obtained by branching on the variable with the biggest right gains, i.e. corresponding to the biggest items (w.r.t. $w_i$), and finally by branching on the dummy variable with gains $(C, C)$.
Consider a tree $T$ that closes the gap $G$.
Observe that each leaf of $T$ has at least depth $n-1$, and at most $n$.
In fact, there is a one-to-one mapping between (not necessarily feasible) solutions of the knapsack instance and nodes at depth $n-1$.
Furthermore, there is a one-to-one mapping between \emph{feasible} solutions of the knapsack instance and \emph{leaves} at depth $n-1$.
Therefore, since an infeasible knapsack instance would yield a perfect tree, with $2^{n + 1} -1$ nodes, each feasible solution decreases the number of nodes by $2$ exactly.
If there are at least $K$ feasible solutions to \countkp{}, there is a tree with at most $2^{n + 1} -1 -2K$ nodes for their corresponding \gap{} instance, and vice-versa.
\end{proof}

In this proof we use the fact that the domination relation, formally defined in Section \ref{sec:scoring}, is a total order on the set of variables, and thus which variables to branch on to obtain a minimum-size tree is trivial.
As a result, we actually prove Theorem \ref{th:sharp-p-hard} for a special case of the \gap{} problem, where at a given level, a single variable is branched on at every node.
Similar to \onevar{}, this special case is therefore a tree measurement problem.

Theorem \ref{th:sharp-p-hard} in particular implies, under the widely believed conjecture that the polynomial hierarchy PH \cite{stockmeyer77a} is proper to the second level, that \gap{} is neither in NP nor in co-NP.
This is because $PH \subseteq P^{\#P}$ \cite{toda91a}.

\section{Using the abstract model for scoring branching candidates}\label{sec:scoring}
\subsection{Applicability of our results to rational numbers}
The results we present in Section \ref{sec:onevar} through \ref{sec:gap} suppose integer gains and gap.
This does not hold in general when solving MIP instances, where numbers are encoded as rationals.
It is however sufficient to notice that in \onevar{}, \nvar{} or \gap{}, if some data is rational, there exists a scaling factor $q \in Q_+$, such that, if the gap and all gains are multiplied by $q$, all data become integer.
Indeed, in the abstract setting described in Section \ref{sec:defs}, which holds for \onevar{}, \nvar{} or \gap{}, suppose a \bb{} tree closes a gap $G$, and all gains are multiplied by $q$, then the same tree closes a gap $q \times G$, since the gap closed at each node is the sum of all gains along the path from the root to the leaf.

Furthermore, given a variable with rational gains $(l, r) \in \Q_+^2$, and a scaling factor $q \in \Q_+$, such that the scaled variable has gains $(q \times l, q \times r) \in Z_+^2$, then the ratio $\varphi$ of the scaled variable can be computed as described in Section \ref{sec:onevar}, and we have
\begin{align*}
 \varphi^{q r} - \varphi^{q (r - l)} - 1 = 0 \Leftrightarrow (\varphi^q)^r - (\varphi^q)^{(r - l)} - 1 = 0 ,
\end{align*}
therefore the ratio of variable $(l, r)$ is given by $\varphi^q$.

Finally, the proof of Theorem \ref{th:cff} (for the closed-form formula) does not require data to be integer, or even rational, and Corollary \ref{cor:fpeq} (for the fixed-point equation) inherently works on rational numbers.
We implement both formulas in our code without any scaling.

\subsection{The \sfratio{} scoring function}\label{sec:sfratio}
Recall from Section \ref{sec:motiv-onevar} that a scoring function combines the left and right LP gains to score candidate variables.
The variable with the highest score is then branched on.
The \sflinear{} and \sfproduct{} functions are two scoring functions that were presented in Section \ref{sec:motiv-onevar}.
In this section, we introduce the \emph{\sfratio{}} scoring function, based on Definition \ref{def:ratio-onevar} for the \onevar{} problem.
Given two variables $i$ and $j$ and their respective ratios $\varphi_i$ and $\varphi_j$, it selects the variable with the smallest ratio.

To help analyze the behavior of the \sfratio{} scoring function, we formalize the concept of domination for two variables.
\begin{definition}
\label{def:variable_domination}
 Given two variables $(l_1, r_1)$ and $(l_2, r_2)$, we say that $(l_1, r_1)$ dominates $(l_2, r_2)$ if and only if $l_1 \geq l_2$ and $r_1 \geq r_2$, with at least one of the inequalities strict.
\end{definition}

\begin{proposition}
Suppose variable $(l_1, r_1)$ dominates variable $(l_2, r_2)$. Then,
\begin{enumerate}
\item Both the \sfproduct{} and the \sflinear{} function assign a better score to $(l_1, r_1)$ than to $(l_2, r_2)$. \label{enum:domin:funcs}
 \item In the \nvar{} or \gap{} problem setting, if both $(l_1, r_1)$ and $(l_2, r_2)$ are branched on on a common path, branching on $(l_1, r_1)$ before branching on $(l_2, r_2)$ yields a tree of size no larger than the converse.\label{enum:domin:ts}
\end{enumerate}
\end{proposition}
\begin{proof}
 Part \ref{enum:domin:funcs} is obvious.
 Part \ref{enum:domin:ts} can be proven by observing that if in a \bb{} tree, variable $(l_2, r_2)$ is branched on before $(l_1, r_1)$, then the tree where the variables are switched closes a gap at least as large.
\end{proof}

In other words, the \sflinear{} and the \sfproduct{} scoring functions prefer non-dominated variables, as dominated variables necessarily yield larger trees.
We prove that the \sfratio{} scoring function also exhibits this property:
 \begin{theorem}
 \label{th:domin-ratio}
 Suppose variable $(l_1, r_1)$ dominates variable $(l_2, r_2)$, then $\varphi_1 < \varphi_2$.
 \end{theorem}
\begin{proof}
 Starting with the definition of the characteristic polynomial of variable $j$, we establish:
 \begin{align*}
  0 
  &= p_2(\varphi_2)
  = \varphi_2^{r_2} - \varphi_2^{r_2 - l_2} -1
  = \varphi_2^{r_2}(1 - \varphi_2^{-l_2}) -1\\
  &< \varphi_2^{r_1}(1 - \varphi_2^{-l_1}) -1
  = p_1(\varphi_2).
 \end{align*}
Therefore, by Theorem \ref{th:pol}, $\varphi_2 > \varphi_1$.
\end{proof}
As a consequence, the ratio of a dominated variable does not have to be computed.
Moreover, to perform comparisons between ratios, it is not necessary to compute both of them.
\begin{proposition}
 \label{prop:test-ratio}
Suppose that the ratio $\varphi_1$ of variable $(l_1, r_1)$ has been computed, then, given a variable $(l_2, r_2)$, $\varphi_2 < \varphi_1$ if and only if $p_2(\varphi_1) > 0$.
\end{proposition}
\begin{proof}
 Direct using Theorem \ref{th:pol}.
\end{proof}
Hence a strategy to reduce the number of ratio computations is to compute $\varphi$ for a variable that is believed to be a good candidate for branching (e.g. the best according to the \sfproduct{} function), and use Proposition \ref{prop:test-ratio} to test all other candidates, computing the ratio of a variable only if it is proven to be smaller than the current best one.

The scoring rule that we implement uses the strategy described in Algorithm \ref{alg:sfratio}.
The computation of the ratio at Steps \ref{alg:step:ratio-product} and \ref{alg:step:ratio-update} is done numerically.
We have implemented and tested five different methods, namely the fixed-point iteration method described in Corollary \ref{cor:fpeq}, a simple bisection method, Newton and Laguerre's method \cite[Chapter 9]{press07a}), and a direct method computing \eqref{def:ratio-onevar} using the closed-form formula described in Theorem \ref{th:cff} for a large enough gap.
Iterative methods are initialized with starting point $\sqrt[r]{2}$, which is a proven lower bound (see Proposition \ref{prop:phibounds}).
Experiments have shown that if $\frac{r}{l} \leq 100$, Laguerre's method was the most effective, otherwise the fixed-point iteration method should be used.
Following this rule, computing the ratio of a variable takes around 20 milliseconds of CPU time on a modern computer.
\begin{algorithm}
\caption{Implementation of the \sfratio{} scoring function.}
\label{alg:sfratio}
\begin{algorithmic}[1]
\State Filter out the branching candidates with dominated gains. \Comment Theorem \ref{th:domin-ratio}
\State Compute the ratio $\varphi^*$ of the best variable according to the \sfproduct{} function. \label{alg:step:ratio-product}
\ForAll{remaining branching candidates $i$}
   \If {$p_i(\varphi^*) < 0$}\Comment Proposition \ref{prop:test-ratio}
      \State Compute the ratio $\varphi_i$ and set $\varphi^* = \varphi_i$. \label{alg:step:ratio-update}
   \EndIf
\EndFor
\State \textbf{return} the variable with ratio $\varphi^*$.
\end{algorithmic}
\end{algorithm}

Observe that if the ratio and one of the left (resp. right) gain are fixed, then the right (resp. left) gain can be computed analytically:
\begin{proposition}
 Given a variable $(l, r)$ and its ratio $\varphi$, then
 \begin{align*}
  r = - \frac{\log(1 - \varphi^{-l})}{\log\varphi} && l=- \frac{\log(1 - \varphi^{-r})}{\log\varphi}.
 \end{align*}
\end{proposition}
\begin{proof}
 Direct using the definition of the polynomial $p$.
\end{proof}
Figure \ref{fig:isoscores} has been generated using this property.
The leftmost point (corresponding to variable $(100, 100)$) serves as a reference, and each curve is such that its points have the same score according to one scoring function: the red curve corresponds to the \sflinear{} function with default $\mu$, blue to the \sfproduct{}, and green to the \sfratio{}.
Unsurprisingly, the red curve is a line and the blue curve is a parabola.
Note how the \sfproduct{} function and the \sfratio{} function match very closely when $l$ and $r$ are close to each other.
This essentially means that if all variables have roughly equal left and right gains, the \sfproduct{} and \sfratio{} functions behave similarly.
If this is not the case, then the \sfratio{} function prefers variables with unequal left and right gains compared to the \sfproduct{} function.

\begin{figure}
\centering
  \begin{tikzpicture}[scale=.7]
  \begin{axis}[xlabel={left gain}, x dir = reverse, ylabel={right gain},
legend entries={\sfratio{}, \sflinear{}, \sfproduct{}},
  ]
  \addplot[green] table[y index = 1] {\fixedscoretofive};
  \addplot[red] table[y index = 2] {\fixedscoretofive};
  \addplot[blue] table[y index = 3] {\fixedscoretofive};
  \end{axis}
 
  \end{tikzpicture}
  \caption{Right gains depending on the left gain such that the score is constant.}
  \label{fig:isoscores}
\end{figure}
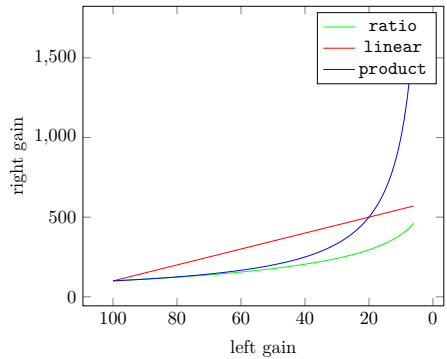

\subsection{The \sfonevar{} scoring function}\label{sec:sfonevar}
Ranking the variables by ratio only makes sense for large enough gaps.
Consider again the example given in Section \ref{sec:motiv-onevar}, and suppose that variables $(10, 10)$ and $(2, 49)$ are the only candidates for branching.
The \sfratio{} function would score variable $(2, 49)$ higher, independent of the gap.
However, in the \onevar{} setting, if the gap is no larger than 40 then $(10, 10)$ is the better variable.
We therefore define a more refined scoring function that takes the gap into account.
Given the gap at a given node, this scoring rule ranks the variables according to their \onevar{} tree-sizes.
We therefore call this scoring function the \sfonevar{} function, for ``Single Variable Tree-Size''.
Algorithm \ref{alg:sfonevar} describes the implementation of the \sfonevar{} function.
We introduce a parameter $D$, which determines whether an approximation of the tree-size should be used, depending on the smallest depth of a leaf in the \onevar{} \bb{} tree, i.e. $\lceil\frac{G}{r_i}\rceil$, given a variable $(l_i, r_i)$ and a gap $G$. 
Parameter $D$ effectively bounds the number of terms to add in the closed-form formula \eqref{eq:onevar-cff} to compute the \onevar{} tree-size.
If $\lceil\frac{G}{r_i}\rceil > D$ (i.e. the gap $G$ is such that the formula \eqref{eq:onevar-cff} would require computing more than $D$ terms), then the exact tree-size is computed for a smaller gap, namely $r_i D$, and an approximation of the tree-size $t(G)$ is computed approximately using the ratio.
In our implementation we have set $D=100$.
Deactivating Step \ref{alg:step:filterdomin} (i.e. not filtering out dominated variables) increases the average running time of a MIP solver implementing Algorithm \ref{alg:sfonevar} as a scoring function in \bb{} by $4\%$ (in the setup of Section \ref{sec:mipexpe-bench}).
Deactivating Step \ref{alg:step:handle-approx} (i.e. not restricting the evaluation of formula \eqref{eq:onevar-cff} to ``small'' trees) increases the average running time by $11\%$.
\begin{algorithm}
\caption{Implementation of the \sfonevar{} scoring function.}
\label{alg:sfonevar}
\begin{algorithmic}[1]
\State Compute the absolute gap $G$ at the current node.
\If {$G = \infty$} \Comment e.g. no primal solution is known.
\State \textbf{return} the variable selected by Algorithm \ref{alg:sfratio}. \Comment Theorem \ref{th:nvar-ratio}
\EndIf
\State Filter out the branching candidates with dominated gains. \label{alg:step:filterdomin}
\ForAll{remaining branching candidates $i$ with gains $(l_i, r_i)$}
   \State Let $d = \lceil \frac{G}{r_i} \rceil$. \Comment The minimum depth of the \onevar{} tree
   \If {$d= \infty$}
      \State Set the tree-size $t_i(G) = \infty$.
   \Else
      \State Let $\tilde{G} = r_i * \min(d, D)$. \label{alg:step:handle-approx} \Comment $D$ is a parameter
      \State Compute the \onevar{} tree-size $t_i(\tilde{G})$ using formula \eqref{eq:onevar-cff}.
      \If {$\tilde{G} < G$}
	  \State Compute the ratio $\varphi_i$.
	  \State Compute the \onevar{} tree-size $t_i(G) \approx \varphi_i^{G - \tilde{G}}t_i(\tilde{G})$. \Comment see \eqref{eq:onevar-ts-approx}
      \EndIf
   \EndIf
\EndFor
\If {all variables $i$ have $t_i(G) = \infty$}
   \State \textbf{return} the variable selected by Algorithm \ref{alg:sfratio}.
\Else
   \State \textbf{return} the variable $i$ with minimum $t_i(G)$.
\EndIf
\end{algorithmic}
\end{algorithm}

Note that for a large enough gap, the \onevar{} function is equivalent to the \sfratio{} function, i.e. the rankings of variables are equal.

We next evaluate the \sfratio{} and \onevar{} scoring function in simulations as well as MIP benchmarks.

\section{Experimental results}\label{sec:expresults}
Numerous MIP solving algorithms are developed and parametrized by a trial-and-error loop of experiments on MIP benchmarks.
The major drawback of this procedure is evidently the computational burden it involves.
Another significant flaw is that the instances used to develop the algorithm are often the ones used to benchmark it.
This may lead to \emph{overfitting}, which means that the algorithm will perform much better on these instances than on general ones.
The magnitude of this problem increases with the number of parameters of a method.
In an effort to mitigate these issues, we carry out simulations on a large set of random instances in Section \ref{sec:simresults}.
The results we obtain on these simulated instances confirm the improvements we achieve on standard MIP benchmarks in Section \ref{sec:mipexpe}.

\subsection{Numerical simulations}\label{sec:simresults}
In order to evaluate the performance of the \sflinear{}, \sfproduct{} and \sfratio{} scoring functions detailed in Section \ref{sec:scoring}, we run simulations in which these scoring functions are used to select variables on \nvar{} and \gap{} instances.
The \gap{} problem models the \bb{} algorithm more closely than the \nvar{} problem, but in practice it is only possible to obtain exact tree-sizes for the latter problem.
We do not run simulations on the \sfonevar{} function, as it requires computing tree-sizes for all variables and each possible gap.

We evaluate the scoring functions on synthetically generated instances.
The variable gains are generated in one of four different ways:
\begin{itemize}
 \item \textbf{Balanced} [B] Both left and right gains are integers uniformly drawn in the interval $[1, 1000]$ (if necessary, gains are switched to ensure that the left gain is less than or equal to the right gain).
 \item \textbf{Unbalanced} [U] The left (resp. right) gain is an integer uniformly drawn in the interval $[1, 500]$ (resp. $[501, 1000]$).
 \item \textbf{Very Unbalanced} [V] The left (resp. right) gain is an integer uniformly drawn in the interval $[1, 250]$ (resp. $[251, 1000]$).
 \item \textbf{Extremely Unbalanced} [X] The left (resp. right) gain is an integer uniformly drawn in the interval $[1, 100]$ (resp. $[101, 1000]$).
\end{itemize}
We will refer to these different ways of generating variables as different \emph{data types}.

\subsubsection{Simulation results for the \nvar{} problem}
Recall that in the \nvar{} problem, one variable can be used an arbitrary number of times.
In this section, two types of methods are evaluated.
The first type correspond to \emph{scoring functions} that can be or are actually implemented in MIP solvers to choose which variable to branch on (e.g. the \sfproduct{} function).
The second type are \emph{problem solving algorithms}, which use knowledge of the \nvar{} problem, and cannot directly be implemented in MIP solvers (e.g. an exact recursive algorithm).
The methods of the first type are the ones we are really interested in, while the methods of the second type help evaluate the first ones.
We compute the tree-sizes that result from the following methods, listed by order of increasing exactness.
\begin{itemize}
 \item \sflinear{} (with different values of the parameter $\mu$), \sfproduct{} and \sfratio{} scoring functions.
 Since the ranking of variables obtained by scoring functions only depends on the left and right gains, the same variable is branched on at every node, which means the tree-size obtained by selecting variables using a scoring function corresponds to the tree-size of the \onevar{} tree for the variable with the best score.
 \item Lower Bound (LB) for scoring functions: it is the minimum tree-size of all \onevar{} instances, for all variables in the input of the \nvar{} instance.
 \item The exact recursive function, as described in Section \ref{sec:nvar}: this method produces a minimum-size tree for the instance.
\end{itemize}
Each branching strategy in this list necessarily produces trees with fewer nodes than those listed higher in the list.
Note that the \sflinear{}, \sfproduct{} and \sfratio{} scoring functions are listed together, as it is \textit{a priori} not clear whether any of the scoring function dominates the others.
These are the scoring functions we are primarily interested in.

The other parameters have been chosen as follows:
\begin{itemize}
 \item The number of variables for each instance is $100$. Increasing the number of variables is possible, however the chances of producing a variable that dominates all others would increase.
 \item The gap is set at $G=10^5$. Increasing $G$ by an order of magnitude may cause the tree-sizes to exceed the maximum number that can be encoded by a double-precision floating-point representation ($\approx 10^{307}$).
 \item The number of instances for each of the B, U, V and X type of variables is 100.
\end{itemize}

Table \ref{tab:simulation-nvar} provides two different performance measures for each type of data and each branching strategy.
The first performance measure, ``t-s'', is the percentile change in geometric means of the tree-sizes compared to the (exact) minimum tree-size.
For instance, the number 26.2 (at the top-left of the table) indicates that this branching strategy yielded trees with, on average, $26.2\%$ more nodes than the minimum-size trees.
For the cases where the difference is in orders of magnitude, only the order is indicated.
Note that the tree-sizes obtained in the experiments vary from around $10^{30}$ for the balanced data to up to $10^{300}$ for extremely unbalanced data.
The second performance measure, ``wins'', is the number of times a scoring function produces the smallest tree-size among the \sflinear{}, \sfproduct{} and \sfratio{} functions (LB is excluded as it necessarily produces smaller tree-sizes).
In case of a tie, multiple scoring functions win.

First, observe that the performance (w.r.t. tree-size) of branching strategies generally decreases as the left and right gains of the variables become less balanced.
For some values of $\mu$, the \sflinear{} function performs extremely badly.
More interestingly, it appears that the best value of $\mu$ heavily depends on the balance of data.
The value $\mu=\frac{1}{2}$ is better for balanced data (B), while for extremely unbalanced data (X), $\mu=\frac{1}{6}$ performs best, in a tie with the \sfratio{} function.
For this type of data, the tie may be explained by Figure \ref{fig:isoscores}, where the \sflinear{} and \sfratio{} functions tend to pick similar variables for very unbalanced gains.
Similarly, the tie between the \sfproduct{} function and the \sfratio{} function for balanced data may be explained by the fact that they pick similar variables for balanced gains.
The \sfratio{} function wins on 397 out of 400 instances, and is very close to the theoretical bound LB that these scoring functions can achieve.
Note that both the \sfratio{} function and LB are within $7\%$ of the actual minimum-tree size, even for the most unbalanced data.

\begin{table}
\centering
\begin{footnotesize}
\begin{tabular}{c|c|ccccc|c|c|c}
 \multirow{2}{*}{Data} & & \multicolumn{5}{c|}{\sflinear{}($\mu$)} & \multirow{2}{*}{\sfproduct{}} & \multirow{2}{*}{\sfratio{}} & \multirow{2}{*}{LB}\\
& $\mu =$ & $0$ & $\frac{1}{6}$ & $\frac{1}{3}$ &  $\frac{1}{2}$ & $1$ & & &  \\
 \hline
 \multirow{2}{*}{B} & t-s & 26.2 & 8.23 & 6.23 & 3.10 & $10^{17}$ & \textbf{2.84} & \textbf{2.84} & 2.83 \\
  & wins & 88 & 93 & 95 & 99 & 8 & \textbf{100} & \textbf{100} &\\
  \hline
 \multirow{2}{*}{U} & t-s& $10^{08}$ & 96.79 & 6.07 & 8.43 & $10^{23}$ & 6.60 & \textbf{3.35} & 3.14 \\
  & wins & 11 & 67 & 93 & 90 & 10 & 93 & \textbf{98} &\\
  \hline
 \multirow{2}{*}{V} & t-s& $10^{20}$ & 57.93 & 24.85 &  282.28 & $10^{23}$ & 29.82 & \textbf{5.75} & 5.75 \\
  & wins & 19 & 77 & 86 & 65 & 14 & 83 & \textbf{99} &\\
  \hline
 \multirow{2}{*}{X} & t-s&  $10^{47}$ & \textbf{6.07} & 731.44 & $10^5$ & $10^{31}$ & 93.86 & \textbf{6.07} & 6.07 \\
  & wins & 7 & \textbf{100} & 72 & 50 & 13 & 82 & \textbf{100} &\\
\end{tabular}
\end{footnotesize}
\caption{Simulation results on \nvar{}.}
\label{tab:simulation-nvar}
\end{table}

Additional results with smaller gaps are given in Appendix \ref{app-sec:additional_simulations}.

\subsubsection{Simulation results for the \gap{} problem}\label{sec:simresults-gap}
The experiments we carry out in this section are significantly different than those we did on \nvar{}: indeed, in the \gap{} problem, each variable can only be used a given number of times on a path from the root to a leaf.
As a consequence, we are not able to compute the exact minimum tree-sizes, as the state space of the recurrence equation \eqref{eq:gap-t} is too large.
For simplicity, and to ensure that the results in this section would be different than those from Section \ref{sec:simresults}, we have set multiplicities to 1.
Note that the algorithms that compute the tree-sizes for scoring functions are also harder to solve than in the previous section, and the gaps $G$ used in our instances are therefore smaller.
The parameters of these experiments are presented below.
\begin{itemize}
 \item As in the previous section, the number of variables for each instance is $100$, and the number of instances for each of the B, U, V and X type of variables is also 100.
 \item The gap takes different values for each type of data. Indeed, while a gap of e.g. $G=5000$ can be closed with a few hundred nodes in the balanced case, it can only be closed in half of the instances for extremely unbalanced data.
 We therefore adapt the gap to the data.
 In each case, the minimum gap tested is chosen such that the tree-sizes are at least 100.
 The maximum gap tested is chosen such that the depth of the trees is close to $100$, i.e. the number of variables, but also such that all 100 instances can be solved.
 The average between the minimum and the maximum gap is also tested.
\end{itemize}
Table \ref{tab:simulation-gap} presents the simulation results, in a format similar to Table \ref{tab:simulation-nvar}.
Since we do not know the minimum tree-sizes in these experiments, the reference is the \sfproduct{} function, as it is the state of the art.
As a consequence, negative changes may (and do) occur.
The observation we made in the previous experiments on \nvar{}, which is that the \sfratio{} function benefits from unbalancedness, clearly carries over.
Furthermore, for a given type of data, it also profits from an increase in the gap.
An approximate $7\%$ decrease in the number of nodes occurs when these two facts are taken together.
Moreover, the \sfratio{} function has, by far, the largest number of wins.
Note that the \sflinear{} function outperforms the \sfratio{} or the \sfproduct{} function in some cases.

\begin{table}
\centering
\begin{footnotesize}
\begin{tabular}{c|c|c|ccccc|c|c}
\multirow{2}{*}{Data}& \multirow{2}{*}{Gap} & & \multicolumn{5}{c|}{\sflinear{}($\mu$)} & \multirow{2}{*}{\sfproduct{}} & \multirow{2}{*}{\sfratio{}}\\
&& $\mu =$ & $0$ & $\frac{1}{6}$ & $\frac{1}{3}$ &  $\frac{1}{2}$ & $1$ & &  \\
\hline
\hline
\multirow{6}{*}{B}&\multirow{2}{*}{5000}& t-s & 1.94 & 0.55 & \textbf{-0.15} & 0.41 & 235.83 & 0.00 & -0.03\\
&& wins & 67 & 77 & \textbf{85} & 76 & 0 & 82 & 83\\
\cline{2-10}
&\multirow{2}{*}{15000}& t-s & 21.34 & 9.37 & 1.18 & 2.36 & 1051.56 & \textbf{0.00} & 0.06\\
&& wins & 2 & 3 & 17 & 14 & 0 & 41 & \textbf{45}\\
\cline{2-10}
&\multirow{2}{*}{25000}& t-s & 57.70 & 19.45 & 0.89 & 19.23 & 1742.11 & 0.00 & \textbf{-0.96}\\
&& wins & 0 & 0 & 7 & 1 & 0 & 28 & \textbf{64}\\
\hline
\hline
\multirow{6}{*}{U}&\multirow{2}{*}{5000}& t-s & 50.86 & 10.63 & \textbf{-0.02} & 4.27 & 314.73 & 0.00 & 0.17\\
&& wins & 0 & 2 & 45 & 3 & 0 & \textbf{48} & 44\\
\cline{2-10}
&\multirow{2}{*}{12500}& t-s & 86.38 & 30.30 & 0.75 & 19.99 & 1562.15 & 0.00 & \textbf{-0.74}\\
&& wins & 0 & 0 & 11 & 0 & 0 & 29 & \textbf{60}\\
\cline{2-10}
&\multirow{2}{*}{20000}& t-s & 71.91 & 22.99 & 0.30 & 61.17 & 3874.17 & 0.00 & \textbf{-2.29}\\
&& wins & 0 & 0 & 2 & 0 & 0 & 14 & \textbf{84}\\
\hline
\hline
\multirow{6}{*}{V} &  \multirow{2}{*}{5000}& t-s & 257.39 & 4.76 & 6.50 & 54.40 & 282.93 & 0.00 & \textbf{-0.91}\\
&& wins & 0 & 1 & 0 & 0 & 0 & 35 & \textbf{65}\\
\cline{2-10}
&  \multirow{2}{*}{7500}& t-s & 428.67 & 11.88 & 14.85 & 102.26 & 425.88 & 0.00 & \textbf{-2.29}\\
&& wins & 0 & 0 & 0 & 0 & 0 & 25 & \textbf{75}\\
\cline{2-10}
&\multirow{2}{*}{10000}& t-s & 528.58 & 24.38 & 25.76 & 135.77 & 525.01 & 0.00 & \textbf{-5.81}\\
&& wins & 0 & 0 & 0 & 0 & 0 & 4 & \textbf{96}\\
\hline
\hline
\multirow{6}{*}{X} & \multirow{2}{*}{2000}& t-s & 445.42 & 8.29 & 42.80 & 63.76 & 89.95 & \textbf{0.00} & 3.15\\
&& wins & 0 & 3 & 0 & 0 & 0 & \textbf{76} & 28\\
\cline{2-10}
& \multirow{2}{*}{3000}& t-s & 903.02 & 6.72 & 47.56 & 73.37 & 107.32 & 0.00 & \textbf{-1.16}\\
&& wins & 0 & 0 & 0 & 0 & 0 & 43 & \textbf{58}\\
\cline{2-10}
& \multirow{2}{*}{4000}& t-s & 1537.76 & 2.47 & 41.26 & 68.27 & 105.88 & 0.00 & \textbf{-6.90}\\
&& wins & 0 & 1 & 0 & 0 & 0 & 2 & \textbf{97}\\
\end{tabular}
\end{footnotesize}
\caption{Simulation results on \gap{}.}
 \label{tab:simulation-gap}
\end{table}

A table similar to Table \ref{tab:simulation-gap} for the \nvar{} problem is given in Appendix \ref{app-sec:additional_simulations}.

Since the \sfproduct{} function seems to perform slightly better for small gaps, and the \sfratio{} improves with larger gaps, the idea of changing the scoring function depending on the gap naturally arises.
Table \ref{tab:simulation-gap-hybrid} provides the results of this hybrid function for the maximum gaps considered for each data type.
The notation is the same as in Table \ref{tab:simulation-gap} (including the fact that the \sfproduct{} function is taken as a reference).
The hybrid \sfratio{}-\sfproduct{} function is parametrized by $h$, the height at which the scoring function used at a node switches from \sfratio{} to \sfproduct{} (the height of a node is the length of the longest path to reach a leaf from this node).
For example, $h=10$ means that at nodes that would have a height of $10$ or less with the \sfratio{} function, the \sfproduct{} function is used instead.
As a consequence, $h=0$ corresponds to the \sfratio{} function, and $h=100$ to the \sfproduct{} function (since there are $100$ variables, the depth of a leaf is at most $100$).

First, observe that we do obtain a slight reduction in the number of nodes for any value of $h$ (except for data U and $h=50$).
Second, note that the optimal value for the height parameter increases as the data becomes less balanced.
Third, observe how for each data type, the relative change compared to the \sfproduct{} function is unimodal, which suggests a low variability of the results and backs up our analysis.

\begin{table}
\centering
\begin{footnotesize}
\begin{tabular}{c|c|ccccccc}
 Data & & (\sfratio{}) & \multicolumn{5}{c}{Hybrid \sfratio{}-\sfproduct{}($h$)} & (\sfproduct{})  \\
(\& Gap) & $h=$ & 0 & 10 & 20 & 30 & 40 & 50 & 100 \\
\hline
B & t-s & -0.96 & \textbf{-1.31} & -0.49 & -0.02 & 0.00 & 0.00 & 0.00\\
(25000)& wins & 6 & \textbf{71} & 28 & 13 & 12 & 12 & 12\\
\hline
U & t-s & -2.29 & -2.29 & \textbf{-3.37} & -1.01 & -0.05 & 0.01 & 0.00\\
(20000) & wins & 0 & 7 & \textbf{79} & 14 & 0 & 0 & 0\\
\hline
V & t-s & -5.81 & -5.81 & -6.18 & \textbf{-8.38} & -4.40 & -0.46 & 0.00\\
(10000) & wins & 0 & 0 & 5 & \textbf{81} & 14 & 0 & 0\\
\hline
X & t-s & -6.90 & -6.95 & -7.38 & -9.03 & \textbf{-10.57} & -8.27 & 0.00\\
(4000) & wins & 0 & 0 & 0 & 14 & \textbf{73} & 13 & 0\\
\end{tabular}
\end{footnotesize}
\caption{Simulation results on \gap{} using a hybrid \sfratio{}-\sfproduct{} function, parametrized by a height $h$.}
 \label{tab:simulation-gap-hybrid}
\end{table}

\subsection{Experiments on MIP instances}\label{sec:mipexpe}
We have modified the solver SCIP 3.1.1 \cite{achterberg09b} to use the \sfratio{} or the \sfonevar{} function, alternatively to the default \sfproduct{} function.
Besides the changes necessary for the implementation of these new scoring functions, no other change to SCIP has been made.
In particular, the decision to use strong branching or pseudo-costs on a given variable and at a given node is unchanged.
Furthermore, note that in SCIP, the selection of the branching variable does not solely rely on the \sfproduct{}.
There are a number of ``soft'' tie-breakers that can come into play if the \sfproduct{} scores of multiple variables are almost equal.
This is referred to as hybrid branching \cite{achterberg09a}.
If at a given node, the variable that the hybrid branching rule selects is not the one that maximizes the \sfproduct{} function, then we keep that variable and do not use the \sfratio{} or \sfonevar{} function.
Finally, for the \sfratio{}, we use the parameter $h=10$ as defined in Section \ref{sec:simresults-gap}, following the results of Table \ref{tab:simulation-gap-hybrid}.
This essentially means that whenever a node is believed to be ``close'' to the leaves, the \sfproduct{} function is used rather than the \sfratio{} function.

\subsubsection{Benchmark instances}\label{sec:mipexpe-bench}
The benchmark test set comprises the instances from MIPLIB 3.0, MIPLIB 2003 \cite{achterberg06a}, and MIPLIB 2010 Benchmark \cite{koch11a}.
These instances are the state of the art in MIP benchmarking.
The instances that are at the time of writing classified as ``open'' have been removed, which leaves 159 instances.
For feasible instances, the optimal value is provided as a primal bound, and primal heuristics are disabled (for all instances).
Cuts after the root node are also disabled to reduce performance variability \cite{lodi13a}.
The disconnected components presolver is disabled as it can lead to different transformed problems, and thus possibly increased variability.
Furthermore, for each instance, ten different seeds are used to create random permutations of the input (using SCIP's internal procedure).
A time limit of two hours is specified.
In this setting, the testbed comprises 1590 instances, and the total running time that this experiment required on a single machine is 4 months.

Three different scoring functions are tested.
The first one is the \sfproduct{} function, which is the default in SCIP.
The second one is the \sfratio{} function, defined in Section \ref{sec:sfratio}.
The third one is the \sfonevar{} function defined in Section \ref{sec:sfonevar}.

Table \ref{tab:mipexpe-bench} gives the result for these three scoring functions, with \sfproduct{} as the reference.
Each line corresponds to the 10 permutations of each instance of the test set.
There are three columns for each scoring function.
The columns of the \sfproduct{} scoring function gives absolute performance measures, while the two others give measures relative to \sfproduct{}.
For \sfproduct{}, the first column provides the number of permutations solved, and, for the other two functions, the difference from the reference.
Note that all 10 permutations of a given instance may not be solved the same number of times depending on the scoring function used.
For each instance, we determine for each scoring function the number of permutations solved.
Let $N$ be the minimum and note that $N$ can be as large as 6 and we could still have no single permutation solved by all three scoring functions.
After having determined $N$ for a given instance, we take the $N$ best results for each scoring function and compute the arithmetic averages.
These averages of time (in seconds) and nodes are directly displayed for the \sfproduct{} function.
In the nodes column, we use the letters \emph{k} and \emph{m} as a shorthand for thousands and millions.
For the \sfratio{} and \sfonevar{} functions, we display the ratio of averages for both time and nodes, compared to the \sfproduct{}.
Considering the $N$ best permutations rather than the intersection of the sets of solved permutations enables more data to be used, and reflects in the averages the fact that some scoring functions solve more permutations, giving them a fair advantage.

The instances for which at least one permutation can be solved at the root node by any setting, or for which no permutation is solved by all settings (i.e. $N=0$), are not displayed.
We similarly exclude instances that can be solved in less than a second by at least one setting and permutation.
Note that these instances are accounted in the total number of instances solved, but not in the different average measures.

The total, geometric mean and shifted geometric mean (with shifts 10 and 100 for time and nodes, respectively) \cite[p. 321]{achterberg07a}  are provided at the bottom of the table.
These are not computed on the averages given at each line, but on each value that is used to compute the line averages (i.e. an instance solved 10 times by all settings counts more than an instance solved fewer than 10 times by all settings).

Function \sfratio{} solves marginally more instances than the \sfproduct{} function.
Both functions \sfratio{} and \sfonevar{} slightly outperform \sfproduct{}, in terms of time and nodes used.
The difference in performance is especially apparent when considering the total resources used, as it mostly reflects the performance on large instances, on which the scoring functions we introduce perform better.
In Section \ref{sec:mipexpe-tree}, we will consider a set of instances that contains instances that experimentally require large \bb{} trees, and on which this observation is comfirmed.

\begin{scriptsize}
\begin{center}
{\setlength{\tabcolsep}{4pt}
\begin{longtable}{c|ccc|ccc|ccc}
\caption{Comparison of scoring function on the benchmark test set (continues in the next page).}
\endfoot
\caption{Comparison of scoring function on the benchmark test set.}
\endlastfoot

     \multirow{2}{*}{Instance} & \multicolumn{3}{|c|}{product} & \multicolumn{3}{|c|}{ratio} & \multicolumn{3}{|c}{svts}\\
 & \# & time (s) & nodes & \# & time (s) & nodes & \# & time (s) & nodes\\
\hline
\hline
\endhead
30n20b8 & \textbf{10} & \textbf{521} & \textbf{10} & \textbf{+0} & {1.02} & {1.49} & \textbf{+0} & {1.01} & {1.34}\\
aflow30a & \textbf{10} & {36} & {2.5k} & \textbf{+0} & {0.98} & {0.99} & \textbf{+0} & \textbf{0.96} & \textbf{0.89}\\
aflow40b & \textbf{10} & {3483} & {210.9k} & \textbf{+0} & \textbf{0.68} & {0.65} & \textbf{+0} & {0.70} & \textbf{0.63}\\
air04 & \textbf{10} & {30} & \textbf{7} & \textbf{+0} & \textbf{0.96} & \textbf{1.00} & \textbf{+0} & {1.03} & \textbf{1.00}\\
app1-2 & \textbf{10} & {845} & \textbf{565} & \textbf{+0} & {0.90} & {1.01} & \textbf{+0} & \textbf{0.89} & {1.14}\\
arki001 & {2} & {4862} & {1.3m} & {+0} & \textbf{0.73} & \textbf{0.74} & \textbf{+1} & {0.85} & {0.80}\\
ash608gpia-3col & \textbf{10} & \textbf{87} & \textbf{5} & \textbf{+0} & {1.04} & \textbf{1.00} & \textbf{+0} & {1.01} & \textbf{1.00}\\
bell5 & \textbf{10} & {1} & \textbf{1.1k} & \textbf{+0} & \textbf{0.99} & \textbf{1.00} & \textbf{+0} & {0.99} & \textbf{1.00}\\
biella1 & \textbf{10} & {217} & \textbf{2.8k} & \textbf{+0} & \textbf{0.98} & {1.06} & \textbf{+0} & {0.99} & {1.00}\\
bienst2 & \textbf{10} & \textbf{288} & {112.6k} & \textbf{+0} & {1.05} & \textbf{0.98} & \textbf{+0} & {1.11} & {1.09}\\
binkar10\_1 & \textbf{10} & \textbf{431} & {135.8k} & \textbf{+0} & {1.09} & {1.03} & \textbf{+0} & {1.01} & \textbf{0.98}\\
blend2 & \textbf{10} & {2} & {222} & \textbf{+0} & \textbf{0.99} & \textbf{1.00} & \textbf{+0} & {1.01} & {1.01}\\
cap6000 & \textbf{10} & \textbf{8} & \textbf{1.8k} & \textbf{+0} & {1.00} & {1.04} & \textbf{+0} & {1.01} & {1.05}\\
csched010 & \textbf{9} & {6035} & {721.2k} & {-1} & \textbf{0.86} & {0.86} & \textbf{+0} & {0.90} & \textbf{0.84}\\
danoint & \textbf{2} & \textbf{7047} & {1.2m} & {-1} & {1.00} & \textbf{0.96} & {-1} & {1.02} & {0.97}\\
dcmulti & \textbf{10} & {2} & \textbf{9} & \textbf{+0} & \textbf{1.00} & \textbf{1.00} & \textbf{+0} & {1.00} & \textbf{1.00}\\
dfn-gwin-UUM & \textbf{10} & \textbf{256} & {62.9k} & \textbf{+0} & {1.03} & {1.07} & \textbf{+0} & {1.05} & \textbf{0.96}\\
eil33-2 & \textbf{10} & {103} & {496} & \textbf{+0} & \textbf{0.97} & {0.99} & \textbf{+0} & {1.11} & \textbf{0.97}\\
eilB101 & \textbf{10} & {134} & {4.4k} & \textbf{+0} & \textbf{0.82} & {0.84} & \textbf{+0} & {0.88} & \textbf{0.76}\\
enlight13 & {0} & {-} & {-} & \textbf{+10} & {-} & {-} & {+0} & {-} & {-}\\
fast0507 & \textbf{10} & {129} & \textbf{494} & \textbf{+0} & \textbf{0.98} & {1.18} & \textbf{+0} & {0.99} & {1.22}\\
fiber & \textbf{10} & \textbf{5} & \textbf{12} & \textbf{+0} & {1.00} & \textbf{1.00} & \textbf{+0} & {1.01} & {1.10}\\
fixnet6 & \textbf{10} & {5} & \textbf{9} & \textbf{+0} & \textbf{1.00} & \textbf{1.00} & \textbf{+0} & {1.01} & {1.02}\\
gesa2-o & \textbf{10} & \textbf{4} & \textbf{4} & \textbf{+0} & {1.01} & \textbf{1.00} & \textbf{+0} & {1.01} & \textbf{1.00}\\
gesa2\_o & \textbf{10} & {4} & \textbf{5} & \textbf{+0} & \textbf{0.99} & \textbf{1.00} & \textbf{+0} & {1.00} & \textbf{1.00}\\
gesa3 & \textbf{10} & {4} & {20} & \textbf{+0} & {0.99} & {1.00} & \textbf{+0} & \textbf{0.99} & \textbf{0.97}\\
gesa3\_o & \textbf{10} & {4} & \textbf{9} & \textbf{+0} & \textbf{0.99} & \textbf{1.00} & \textbf{+0} & {1.00} & \textbf{1.00}\\
glass4 & \textbf{10} & \textbf{29} & \textbf{11.5k} & \textbf{+0} & {1.26} & {1.34} & \textbf{+0} & {1.05} & {1.03}\\
iis-100-0-cov & \textbf{10} & {2286} & {83.6k} & \textbf{+0} & \textbf{0.97} & \textbf{0.99} & \textbf{+0} & {0.97} & {0.99}\\
iis-pima-cov & \textbf{10} & {806} & {6.7k} & \textbf{+0} & {1.00} & \textbf{0.96} & \textbf{+0} & \textbf{1.00} & {0.96}\\
khb05250 & \textbf{10} & {1} & \textbf{7} & \textbf{+0} & {0.99} & \textbf{1.00} & \textbf{+0} & \textbf{0.98} & \textbf{1.00}\\
l152lav & \textbf{10} & \textbf{2} & \textbf{16} & \textbf{+0} & {1.01} & {1.05} & \textbf{+0} & {1.01} & \textbf{1.00}\\
lseu & \textbf{10} & {2} & {1.0k} & \textbf{+0} & \textbf{0.92} & \textbf{0.83} & \textbf{+0} & {1.09} & {1.11}\\
map18 & \textbf{10} & \textbf{462} & {260} & \textbf{+0} & {1.03} & \textbf{0.99} & \textbf{+0} & {1.05} & {1.00}\\
map20 & \textbf{10} & \textbf{422} & {300} & \textbf{+0} & {1.06} & {1.00} & \textbf{+0} & {1.00} & \textbf{0.98}\\
mas74 & \textbf{10} & {1810} & {3.3m} & \textbf{+0} & \textbf{0.87} & \textbf{0.88} & \textbf{+0} & {0.92} & {0.90}\\
mas76 & \textbf{10} & {200} & {484.2k} & \textbf{+0} & \textbf{0.83} & \textbf{0.83} & \textbf{+0} & {0.98} & {0.98}\\
mcsched & \textbf{10} & {268} & {18.1k} & \textbf{+0} & {0.99} & {1.01} & \textbf{+0} & \textbf{0.98} & \textbf{0.95}\\
mik-250-1-100-1 & \textbf{10} & \textbf{1832} & \textbf{829.9k} & \textbf{+0} & {1.18} & {1.17} & \textbf{+0} & {1.30} & {1.31}\\
mine-166-5 & \textbf{10} & \textbf{30} & {356} & \textbf{+0} & {1.01} & \textbf{0.95} & \textbf{+0} & {1.05} & {1.13}\\
mine-90-10 & \textbf{10} & {1228} & {116.9k} & \textbf{+0} & \textbf{0.92} & \textbf{0.96} & \textbf{+0} & {1.13} & {1.18}\\
misc03 & \textbf{10} & {2} & {125} & \textbf{+0} & \textbf{0.98} & {1.02} & \textbf{+0} & {0.98} & \textbf{0.97}\\
misc06 & \textbf{10} & {1} & \textbf{7} & \textbf{+0} & \textbf{0.99} & \textbf{1.00} & \textbf{+0} & {1.01} & \textbf{1.00}\\
misc07 & \textbf{10} & \textbf{45} & {23.8k} & \textbf{+0} & {1.24} & {1.28} & \textbf{+0} & {1.06} & \textbf{0.99}\\
mod008 & \textbf{10} & {5} & \textbf{12} & \textbf{+0} & \textbf{0.99} & {1.07} & \textbf{+0} & {0.99} & \textbf{1.00}\\
mod011 & \textbf{10} & {74} & {886} & \textbf{+0} & \textbf{0.97} & \textbf{0.89} & \textbf{+0} & {0.98} & {0.96}\\
modglob & \textbf{10} & {1} & \textbf{20} & \textbf{+0} & \textbf{1.00} & {1.04} & \textbf{+0} & {1.00} & {1.01}\\
momentum2 & {1} & \textbf{4417} & \textbf{13.6k} & \textbf{+1} & {1.30} & {1.33} & \textbf{+1} & {1.15} & {1.13}\\
msc98-ip & \textbf{10} & {1959} & {6.3k} & \textbf{+0} & {0.90} & {1.00} & \textbf{+0} & \textbf{0.77} & \textbf{0.75}\\
mzzv11 & \textbf{10} & {732} & {233} & \textbf{+0} & \textbf{0.99} & \textbf{0.89} & \textbf{+0} & {1.00} & {0.91}\\
mzzv42z & \textbf{10} & {554} & \textbf{11} & \textbf{+0} & {1.01} & {1.04} & \textbf{+0} & \textbf{0.97} & {1.04}\\
n4-3 & \textbf{10} & \textbf{1213} & {46.6k} & \textbf{+0} & {1.04} & {0.99} & \textbf{+0} & {1.03} & \textbf{0.95}\\
neos-1109824 & \textbf{10} & {395} & {23.2k} & \textbf{+0} & \textbf{0.52} & \textbf{0.37} & \textbf{+0} & {0.93} & {0.85}\\
neos-1396125 & \textbf{10} & {486} & {69.9k} & \textbf{+0} & {1.00} & {1.01} & \textbf{+0} & \textbf{0.81} & \textbf{0.81}\\
neos-476283 & \textbf{10} & {431} & {131} & \textbf{+0} & \textbf{0.88} & {0.99} & \textbf{+0} & {1.00} & \textbf{0.94}\\
neos-686190 & \textbf{10} & \textbf{142} & \textbf{1.8k} & \textbf{+0} & {1.13} & {1.22} & \textbf{+0} & {1.13} & {1.18}\\
neos13 & \textbf{10} & {275} & {12} & \textbf{+0} & \textbf{0.95} & \textbf{0.98} & \textbf{+0} & {1.00} & \textbf{0.98}\\
neos18 & \textbf{10} & \textbf{76} & \textbf{5.3k} & \textbf{+0} & {1.27} & {1.59} & \textbf{+0} & {1.12} & {1.16}\\
net12 & \textbf{10} & \textbf{2011} & \textbf{2.9k} & \textbf{+0} & {1.03} & {1.07} & \textbf{+0} & {1.08} & {1.08}\\
netdiversion & \textbf{10} & {1463} & {13} & \textbf{+0} & {0.97} & \textbf{0.92} & \textbf{+0} & \textbf{0.95} & {1.04}\\
newdano & \textbf{5} & {5636} & {1.3m} & {-1} & {1.02} & \textbf{0.96} & \textbf{+0} & \textbf{0.95} & {0.97}\\
noswot & \textbf{10} & \textbf{573} & \textbf{679.8k} & \textbf{+0} & {1.54} & {1.57} & \textbf{+0} & {1.43} & {1.36}\\
ns1208400 & \textbf{10} & \textbf{414} & \textbf{174} & \textbf{+0} & {1.10} & {1.20} & \textbf{+0} & {1.13} & {1.33}\\
ns1688347 & \textbf{10} & \textbf{81} & {92} & \textbf{+0} & {1.02} & {1.14} & \textbf{+0} & {1.05} & \textbf{0.92}\\
ns1766074 & \textbf{10} & {5740} & {925.7k} & \textbf{+0} & \textbf{0.98} & \textbf{0.98} & \textbf{+0} & {0.99} & \textbf{0.98}\\
ns1830653 & \textbf{10} & {492} & {18.6k} & \textbf{+0} & {0.98} & {1.03} & \textbf{+0} & \textbf{0.91} & \textbf{0.92}\\
nsrand-ipx & {3} & {5004} & {584.7k} & \textbf{+7} & \textbf{0.33} & \textbf{0.30} & {+2} & {0.50} & {0.49}\\
nw04 & \textbf{10} & {2719} & \textbf{10} & \textbf{+0} & \textbf{0.84} & \textbf{1.00} & \textbf{+0} & {0.93} & \textbf{1.00}\\
opm2-z7-s2 & \textbf{10} & {858} & {977} & \textbf{+0} & {1.03} & {1.00} & \textbf{+0} & \textbf{0.92} & \textbf{0.79}\\
p0201 & \textbf{10} & {3} & \textbf{10} & \textbf{+0} & \textbf{0.99} & {1.02} & \textbf{+0} & {1.00} & {1.02}\\
p2756 & \textbf{10} & {4} & \textbf{10} & \textbf{+0} & \textbf{1.00} & \textbf{1.00} & \textbf{+0} & {1.01} & \textbf{1.00}\\
pg5\_34 & \textbf{10} & {1397} & {116.0k} & \textbf{+0} & \textbf{0.93} & \textbf{0.93} & \textbf{+0} & {0.96} & {0.94}\\
pk1 & \textbf{10} & \textbf{123} & \textbf{325.8k} & \textbf{+0} & {1.13} & {1.14} & \textbf{+0} & {1.09} & {1.06}\\
pp08a & \textbf{10} & {2} & \textbf{237} & \textbf{+0} & {1.00} & {1.01} & \textbf{+0} & \textbf{0.98} & {1.01}\\
pp08aCUTS & \textbf{10} & {2} & \textbf{153} & \textbf{+0} & \textbf{0.99} & {1.02} & \textbf{+0} & {0.99} & {1.05}\\
pw-myciel4 & {0} & {-} & {-} & \textbf{+7} & {-} & {-} & {+4} & {-} & {-}\\
qiu & \textbf{10} & {106} & {12.5k} & \textbf{+0} & {1.00} & {1.02} & \textbf{+0} & \textbf{0.98} & \textbf{0.98}\\
qnet1 & \textbf{10} & {11} & \textbf{4} & \textbf{+0} & \textbf{0.98} & \textbf{1.00} & \textbf{+0} & {1.00} & \textbf{1.00}\\
qnet1\_o & \textbf{10} & {6} & \textbf{4} & \textbf{+0} & \textbf{0.97} & \textbf{1.00} & \textbf{+0} & {0.98} & \textbf{1.00}\\
rail507 & \textbf{10} & \textbf{130} & \textbf{507} & \textbf{+0} & {1.07} & {1.17} & \textbf{+0} & {1.08} & {1.19}\\
ran16x16 & \textbf{10} & {544} & {349.3k} & \textbf{+0} & {0.82} & {0.82} & \textbf{+0} & \textbf{0.81} & \textbf{0.78}\\
rd-rplusc-21 & {0} & {-} & {-} & \textbf{+1} & {-} & {-} & {+0} & {-} & {-}\\
reblock67 & \textbf{10} & {250} & {47.9k} & \textbf{+0} & {0.98} & {1.02} & \textbf{+0} & \textbf{0.93} & \textbf{0.94}\\
rmatr100-p10 & \textbf{10} & \textbf{109} & {799} & \textbf{+0} & {1.03} & {1.00} & \textbf{+0} & {1.02} & \textbf{0.93}\\
rmatr100-p5 & \textbf{10} & {140} & {387} & \textbf{+0} & {0.96} & {0.99} & \textbf{+0} & \textbf{0.95} & \textbf{0.99}\\
rmine6 & {3} & \textbf{3540} & \textbf{380.6k} & \textbf{+1} & {1.10} & {1.05} & {-2} & {2.00} & {2.39}\\
rocII-4-11 & \textbf{10} & \textbf{419} & \textbf{4.7k} & \textbf{+0} & {2.43} & {5.91} & \textbf{+0} & {1.27} & {1.80}\\
rococoC10-001000 & {1} & {4677} & {323.9k} & \textbf{+1} & {0.83} & {0.81} & \textbf{+1} & \textbf{0.44} & \textbf{0.50}\\
roll3000 & \textbf{9} & {3781} & {429.3k} & {-2} & {1.15} & {1.18} & {-3} & \textbf{0.84} & \textbf{0.83}\\
rout & \textbf{10} & {111} & {44.6k} & \textbf{+0} & \textbf{0.73} & \textbf{0.73} & \textbf{+0} & {0.94} & {0.94}\\
satellites1-25 & \textbf{10} & \textbf{1590} & {3.4k} & \textbf{+0} & {1.03} & \textbf{0.98} & \textbf{+0} & {1.02} & {1.01}\\
set1ch & \textbf{10} & \textbf{2} & {5} & \textbf{+0} & {1.00} & \textbf{0.92} & \textbf{+0} & {1.00} & \textbf{0.92}\\
sp98ic & \textbf{10} & {4736} & {276.2k} & \textbf{+0} & \textbf{0.49} & \textbf{0.42} & \textbf{+0} & {0.70} & {0.58}\\
sp98ir & \textbf{10} & {71} & \textbf{1.6k} & \textbf{+0} & \textbf{0.97} & {1.16} & \textbf{+0} & {0.98} & {1.15}\\
stein27 & \textbf{10} & {2} & {4.2k} & \textbf{+0} & \textbf{0.98} & \textbf{0.98} & \textbf{+0} & {1.00} & {0.99}\\
stein45 & \textbf{10} & {34} & {51.5k} & \textbf{+0} & \textbf{0.92} & \textbf{0.98} & \textbf{+0} & {0.97} & {0.99}\\
tanglegram1 & \textbf{7} & {1749} & {289} & {-1} & {1.13} & {1.06} & {-2} & \textbf{0.90} & \textbf{0.58}\\
tanglegram2 & \textbf{10} & {10} & \textbf{4} & \textbf{+0} & \textbf{0.98} & \textbf{1.00} & \textbf{+0} & {1.00} & \textbf{1.00}\\
timtab1 & \textbf{10} & \textbf{1184} & \textbf{774.5k} & \textbf{+0} & {1.04} & {1.04} & \textbf{+0} & {1.01} & {1.02}\\
tr12-30 & \textbf{10} & {3644} & \textbf{780.2k} & \textbf{+0} & {1.02} & {1.08} & \textbf{+0} & \textbf{1.00} & {1.02}\\
unitcal\_7 & {9} & {3415} & {32.0k} & \textbf{+1} & \textbf{0.76} & \textbf{0.75} & {+0} & {0.91} & {0.87}\\
vpm2 & \textbf{10} & {2} & \textbf{180} & \textbf{+0} & {1.00} & {1.05} & \textbf{+0} & \textbf{0.98} & {1.01}\\
zib54-UUE & \textbf{3} & \textbf{4914} & \textbf{413.0k} & \textbf{+0} & {1.20} & {1.18} & \textbf{+0} & {1.09} & {1.13}\\
\hline
\hline
Total & {1225} & {705680} & {121.4m} & \textbf{+24} & \textbf{0.92} & \textbf{0.97} & {+1} & {0.94} & {0.97}\\
\hline
Geo. mean & & {98} & {1.2k} & & \textbf{0.98} & {1.00} & & {0.99} & \textbf{0.98}\\
\hline
Sh. geo. mean & & {148} & {2.8k} & & \textbf{0.98} & {1.00} & & {0.99} & \textbf{0.98}\\
\hline

\label{tab:mipexpe-bench}
\end{longtable}
}
\end{center}
\end{scriptsize}

\subsubsection{Instances with large \bb{} trees}\label{sec:mipexpe-tree}
MIPLIB 2010 also has a so-called tree set with 52 instances, which ``contains instances that (empirically) lead to large enumeration trees'' \cite{koch11a}.
These instances have been selected because the scoring functions \sfratio{} and \sfonevar{} defined in Section \ref{sec:scoring} have specifically been designed to solve instances with large \bb{} trees.
We have tested our scoring functions on these instances in the same setup as the experiments of Section \ref{sec:mipexpe-bench}, but we gave a time limit of 12 hours to allow for a more significant number of instances to be solved.
Including permutations, the testbed comprises 520 instances, and required the equivalent of 19 months of running time on a single machine.

Table \ref{tab:mipexpe-tree} provides the result for the tree test set with the same notation as Table \ref{tab:mipexpe-bench} in Section \ref{sec:mipexpe-bench}.
Both functions \sfratio{} and \sfonevar{} largely outperform \sfproduct{}, both in terms of time and nodes, and solve slightly more instances.
Note that some instances in this test set do not require larger trees than those in the benchmark test set (e.g. glass4, ns894788 and pg).
This is probably due to the fact that MIP solvers have made significant progress for these instances since this test set was designed.
\announceOS{}

\begin{scriptsize}
\begin{center}
{\setlength{\tabcolsep}{4pt}
\begin{longtable}{c|ccc|ccc|ccc}
\caption{Comparison of scoring function on the tree test set (continues in the next page).}
\endfoot
\caption{Comparison of scoring function on the tree test set.}
\endlastfoot

     \multirow{2}{*}{Instance} & \multicolumn{3}{|c|}{product} & \multicolumn{3}{|c|}{ratio} & \multicolumn{3}{|c}{svts}\\
 & \# & time (s) & nodes & \# & time (s) & nodes & \# & time (s) & nodes\\
\hline
\hline
\endhead
csched007 & \textbf{10} & {31000} & {5.7m} & \textbf{+0} & \textbf{0.78} & \textbf{0.75} & \textbf{+0} & {0.82} & {0.76}\\
csched008 & \textbf{10} & \textbf{1440} & \textbf{112.4k} & \textbf{+0} & {1.21} & {1.03} & \textbf{+0} & {1.14} & {1.06}\\
glass4 & \textbf{10} & \textbf{28} & \textbf{11.5k} & \textbf{+0} & {1.37} & {1.34} & \textbf{+0} & {1.03} & {1.03}\\
gmu-35-40 & {0} & {-} & {-} & \textbf{+2} & {-} & {-} & {+1} & {-} & {-}\\
k16x240 & \textbf{10} & {23835} & {23.2m} & \textbf{+0} & {0.88} & {0.89} & \textbf{+0} & \textbf{0.86} & \textbf{0.85}\\
neos-1616732 & \textbf{10} & {9182} & {2.5m} & \textbf{+0} & {0.85} & {0.94} & \textbf{+0} & \textbf{0.82} & \textbf{0.84}\\
neos-942830 & \textbf{10} & {9184} & {1.5m} & \textbf{+0} & \textbf{0.95} & {0.95} & \textbf{+0} & {1.00} & \textbf{0.94}\\
neos15 & \textbf{1} & {34352} & {12.3m} & \textbf{+0} & {0.91} & {0.94} & \textbf{+0} & \textbf{0.87} & \textbf{0.91}\\
neos858960 & \textbf{10} & {3997} & {2.8m} & \textbf{+0} & {0.92} & \textbf{0.93} & \textbf{+0} & \textbf{0.92} & \textbf{0.93}\\
noswot & \textbf{10} & \textbf{556} & \textbf{679.8k} & \textbf{+0} & {1.62} & {1.57} & \textbf{+0} & {1.44} & {1.36}\\
ns1766074 & \textbf{10} & {5747} & {925.7k} & \textbf{+0} & \textbf{0.95} & \textbf{0.98} & \textbf{+0} & {0.98} & \textbf{0.98}\\
pg & \textbf{10} & \textbf{24} & {203} & \textbf{+0} & {1.01} & \textbf{0.97} & \textbf{+0} & {1.01} & {1.10}\\
ran14x18 & \textbf{10} & {23039} & {18.4m} & \textbf{+0} & \textbf{0.83} & {0.82} & \textbf{+0} & {0.87} & \textbf{0.81}\\
reblock166 & {4} & {16858} & {1.6m} & \textbf{+2} & {0.98} & {1.00} & \textbf{+2} & \textbf{0.79} & \textbf{0.95}\\
timtab1 & \textbf{10} & \textbf{1142} & \textbf{774.5k} & \textbf{+0} & {1.05} & {1.04} & \textbf{+0} & {1.02} & {1.02}\\
umts & \textbf{10} & {8726} & {1.5m} & \textbf{+0} & {0.69} & {0.71} & \textbf{+0} & \textbf{0.64} & \textbf{0.65}\\
wachplan & \textbf{10} & \textbf{15334} & \textbf{324.2k} & \textbf{+0} & {1.02} & {1.01} & \textbf{+0} & {1.01} & {1.01}\\
\hline
\hline
Total & {155} & {1434131} & {602.1m} & \textbf{+4} & {0.87} & {0.87} & {+3} & \textbf{0.87} & \textbf{0.84}\\
\hline
Geo. mean & & {2879} & {618.6k} & & {0.97} & {0.96} & & \textbf{0.95} & \textbf{0.95}\\
\hline
Sh. geo. mean & & {3021} & {636.4k} & & {0.97} & {0.96} & & \textbf{0.95} & \textbf{0.95}\\
\hline

\label{tab:mipexpe-tree}
\end{longtable}
}
\end{center}

\end{scriptsize}

\section{Conclusions}\label{sec:conclusion}
We developed one of the first models for the branching component of the \BB{} algorithm, over fifty years after its introduction.
We proved that these models are relevant by theoretically establishing new scoring functions that are efficient for MIP solving.
Numerous questions naturally arise regarding these models.

One possible line of investigation relates to the computational complexity of the decision problems we have defined.
For instance, can  \nvar{} be solved in polynomial time (even for two variables)?
Is there an approximation algorithm for the minimization version of \gap{}?

Scoring functions are the only components of \bb{} that we numerically analyze through the theory we develop in this paper.
Are there other elements of \bb{} that can be studied via the current models?
There certainly exist decision problems that can model the \bb{} algorithm more accurately than \gap{} (e.g. if the gains of a variable are not fixed).
Would the analysis of these models deepen our understanding of \bb{}, and lead to additional MIP solving improvements?

\begin{acknowledgements}
The authors would like to thank Eduardo Uchoa for pointing out reference \cite{kullmann09a} and Graham Farr for helping with the discussion at the end of Section \ref{sec:gap-sharp-p-hard}.\thankreviewers{}
This research was funded by AFOSR grant FA9550-12-1-0151 of the Air Force Office of Scientific Research and the National Science Foundation Grant CCF-1415460 to the Georgia Institute of Technology.
\end{acknowledgements}

\bibliography{branching}
\bibliographystyle{plain}

\appendix
\appendixpage 
\section{Proofs}
\subsection{Proof of Theorem \ref{th:onevar-cvg} and Corollary \ref{cor:fpeq}}\label{app-sec:proof-onevar}
 \begin{theorem*}
 When $G$ tends to infinity, both sequences $\sqrt[l]{\frac{t(G+l)}{t(G)}}$ and $\sqrt[r]{\frac{t(G+r)}{t(G)}}$ converge to $\varphi$, which is the unique root greater than $1$ of the equation $p(x)= x^r - x^{r-l} -1=0$. 
 \end{theorem*}
 \begin{proof}
Proposition \ref{prop:phibounds} proves the case $l=r$, therefore we suppose throughout the proof that $r>l$.
  First, we define the notation
  \begin{align*}
   \underline{L} = \liminf_{G\rightarrow \infty} \frac{t(G+l)}{t(G)} && \bar{L} = \limsup_{G\rightarrow \infty} \frac{t(G+l)}{t(G)} \\
   \underline{R} = \liminf_{G\rightarrow \infty} \frac{t(G+r)}{t(G)} && \bar{R} = \limsup_{G\rightarrow \infty} \frac{t(G+r)}{t(G)}.
  \end{align*}
Using the recursive definition \eqref{eq:onevar-t} of $t$ between the first and the second lines, we can easily establish:
\begin{align*}
 \underline{L} &= \left( \limsup_{G\rightarrow \infty} \frac{t(G)}{t(G+l)} \right)^{-1} = \left( \limsup_{G\rightarrow \infty} \frac{t(G-l)}{t(G)} \right)^{-1} \\
 &= \left( \limsup_{G\rightarrow \infty} \frac{t(G) - t(G-r)}{t(G)} \right)^{-1} = \left(1- \limsup_{G\rightarrow \infty} \frac{t(G-r)}{t(G)} \right)^{-1} \\
 &= \left(1- \limsup_{G\rightarrow \infty} \frac{t(G)}{t(G+r)} \right)^{-1} = \left(1- \left(\liminf_{G\rightarrow \infty} \frac{t(G+r)}{t(G)} \right)^{-1} \right)^{-1} \\
 &= (1-\underline{R}^{-1})^{-1}   = 1 + \frac{1}{\underline{R} - 1}. \quad (\text{note that $\underline{R} \geq 2$.})
\end{align*}
Similarly, $\bar{L} = 1 + \frac{1}{\bar{R} - 1}$. 
Next, we obtain
\begin{align*}
 \liminf_{G\rightarrow \infty} \frac{t(G+lr)}{t(G)}
 &= \liminf_{G\rightarrow \infty} \frac{t(G+lr)}{t(G+(l-1)r)}  \frac{t(G+(l-1)r)}{t(G+(l-2)r)} \dots \frac{t(G+r)}{t(G)}\\
 &\geq \liminf_{G\rightarrow \infty} \frac{t(G+lr)}{t(G+(l-1)r)} \liminf_{G\rightarrow \infty} \frac{t(G+(l-1)r)}{t(G+(l-2)r)} \dots  \liminf_{G\rightarrow \infty}\frac{t(G+r)}{t(G)}\\
 &= \left(\liminf_{G\rightarrow \infty}\frac{t(G+r)}{t(G)}\right)^{l} = \underline{R}^l.
\end{align*}
Also, we have
\begin{align*}
 \limsup_{G\rightarrow \infty} \frac{t(G+lr)}{t(G)}
 &= \limsup_{G\rightarrow \infty} \frac{t(G+lr)}{t(G+l(r-1))} \dots \frac{t(G+l)}{t(G)}\\
 &\leq \limsup_{G\rightarrow \infty} \frac{t(G+lr)}{t(G+l(r-1))} \dots  \limsup_{G\rightarrow \infty}\frac{t(G+l)}{t(G)}\\
 &= \left(\limsup_{G\rightarrow \infty}\frac{t(G+l)}{t(G)}\right)^{r} = \bar{L}^r.
\end{align*}
Together, these yield
\begin{align*}
  \bar{L}^r \geq \limsup_{G\rightarrow \infty} \frac{t(G+lr)}{t(G)} \geq \liminf_{G\rightarrow \infty} \frac{t(G+lr)}{t(G)}\geq\underline{R}^l.
\end{align*}
Following the same steps, we can show $\bar{R}^l \geq \underline{L}^r$.
Starting from this inequality, we prove
\begin{align*}
 \bar{R}^\frac{l}{r} \geq \underline{L}= 1 + \frac{1}{\underline{R} - 1} \Rightarrow \underline{R} \geq 1  + \frac{1}{\bar{R}^\frac{l}{r} -1}.
\end{align*}
Likewise, we can show that the inequality $\bar{R} \leq 1  + \frac{1}{\underline{R}^\frac{l}{r} -1}$ holds.

We now introduce two monotonic sequences $\alpha_n$ and $\omega_n$ that respectively bound $\underline{R}$ from below and bound $\bar{R}$ from above, and we prove that they converge to the same limit.
For all non-negative $n$, let $\alpha_n$ and $\omega_n$ be defined as follows:
\begin{align*}
\alpha_n = 
 \begin{cases}
 2 &\text{ if } n = 0\\
  f(\omega_{n-1}) &\text{ otherwise}
 \end{cases}
 &&
 \omega_n = f(\alpha_n)
\end{align*}
where $f(x) = 1  + \frac{1}{x^\frac{l}{r} -1}$ for all $x \in (1, \infty)$.
We first prove by induction that for all non-negative integer $n$, $\alpha_n$ and $\omega_n$ satisfy $\alpha_n \leq \underline{R} \leq \bar{R} \leq \omega_n$.
Proposition \ref{prop:phibounds} ensures that $\alpha_0=2$ is a lower bound on $\underline{R}$.
Suppose that for a given non-negative $n$, the inequality $\alpha_n \leq \underline{R}$ holds.
We prove that $\bar{R} \leq \omega_n$ holds too:
\begin{align*}
\alpha_n \leq \underline{R}
&\Rightarrow 1  + \frac{1}{\alpha_n^\frac{l}{r} -1} \geq 1  + \frac{1}{\underline{R}^\frac{l}{r} -1} \geq \bar{R}\\
&\Rightarrow f(\alpha_n) \geq \bar{R}\\
&\Rightarrow \omega_n \geq \bar{R}.
\end{align*}
The same reasoning also proves that for all non-negative integers $n$, $\omega_n \geq \bar{R}$ implies $\alpha_{n+1} \leq \underline{R}$.
The sequence $\alpha_n$ thus bounds $\underline{R}$ from below, and $\omega_n$ bounds $\bar{R}$ from above.

We now prove that $\alpha_n$ is monotonically increasing.
Consider the inequality:
\begin{align*}
 \alpha_1 > \alpha_0 
 &\Leftrightarrow f(f(2)) > 2 \Leftrightarrow \frac{1}{f(2)^\frac{l}{r} -1} > 1 \Leftrightarrow f(2) < 2^\frac{r}{l}\\
 &\Leftrightarrow 1 + \frac{1}{2^\frac{l}{r} -1} < 2^\frac{r}{l} \Leftrightarrow 2^\frac{l}{r} < 2^\frac{r}{l} (2^\frac{l}{r} -1) \Leftrightarrow r> l.
\end{align*}
Suppose now that for a given non-negative $n$,  $\alpha_n > \alpha_{n-1}$.
This implies that $\omega_n < \omega_{n-1}$:
\begin{align*}
 \omega_n - \omega_{n-1} = \frac{1}{\alpha_n^\frac{l}{r} -1} - \frac{1}{\alpha_{n-1}^\frac{l}{r} -1} < 0 \Leftrightarrow \alpha_n^\frac{l}{r} > \alpha_{n-1}^\frac{l}{r}.
\end{align*}
In turn, for that given $n$, $\omega_n < \omega_{n-1}$ implies $\alpha_{n+1} > \alpha_{n}$.
The sequence $\alpha_{n}$ and $\omega_n$ are thus increasing and decreasing, respectively.
Since they are bounded, each of them converges to one of the solutions of the fixed-point equation $x=f(x)$.
We now prove that there is a unique solution to that equation:
\begin{align*}
x=f(x) &\Leftrightarrow x =  1  + \frac{1}{x^\frac{l}{r} -1} \Leftrightarrow (x^\frac{l}{r} -1)(x-1)=1\\
&\Leftrightarrow x^\frac{l+r}{r} - x^\frac{l}{r} -x =0\\
&\Leftrightarrow p(X) = X^r - X^{r-l} -1 = 0
\end{align*}
where $X = x^{\frac{1}{r}}$.
We establish in Theorem \ref{th:pol} that the polynomial $p$ has a unique root in $(1, \infty)$, hence the fixed-point equation $x=f(x)$ also has a unique solution.
Consequently, it is necessary that both sequences $\alpha_{n}$ and $\omega_n$ converge to this unique fixed point, and thus $\underline{R} = \bar{R}$.
Furthermore, the sequence $\sqrt[r]{\frac{t(G+r)}{t(G)}}$ converges to the root $\varphi > 1$ of the polynomial $p$.

Since we have established that $\underline{L} = 1 + \frac{1}{\underline{R} - 1}$ and $\bar{L} = 1 + \frac{1}{\bar{R} - 1}$, it follows that $\underline{L} = \bar{L}$.
Since $p(\varphi)=0$, $\varphi$ equivalently satisfies $\varphi^{l} = 1 + \frac{1}{\varphi^{r} - 1}$, therefore $\underline{L} = \bar{L} = \varphi^l$.
 \end{proof}
 
\begin{corollary*}
A numerical approximation of $\varphi^r$ is given by the fixed-point iteration
\begin{equation*}
 f(x) = 1  + \frac{1}{x^\frac{l}{r} -1}.
\end{equation*}
with the starting point $x=2$.
\end{corollary*} 
\begin{proof}
Recall the definition of the sequences $\alpha_n$ and $\omega_n$ as given in the proof of Theorem \ref{th:onevar-cvg} (notice that function $f$ has the same definition).
The sequence $f_n$ generated by the fixed-point equation is $(2, f(2), f(f(2)), f(f(f(2))), \dots)$, which is equal to 
$(\alpha_0, \omega_0, \alpha_1, \omega_1, \dots)$.
Formally, the sequence $f_n$ generated by the fixed-point equation satisfies
 \begin{equation*}
 f_n = \begin{cases} \alpha_{\frac{n}{2}} & \text{ if $n$ is even}\\ \omega_{\frac{n+1}{2}} & \text{ if $n$ is odd.} \end{cases}
 \end{equation*}
 In the proof of Theorem \ref{th:onevar-cvg}, we prove that both $\alpha_n$ and $\omega_n$ converge to $\varphi^r$ when $n$ tends to infinity, therefore $f_n$ also converges to $\varphi^r$.
\end{proof}

\subsection{Proof of Theorem \ref{th:nvar-ratio}}\label{app-sec:proof-nvar}
Recall that $z$ is the least common multiple of all $l_i$ and $r_i$.
 \begin{theorem*}
  $\varphi = \min_i \varphi_i$
 \end{theorem*}
\begin{proof}
Let $\alpha$ be such that for all positive integers $x \leq z$,
\begin{equation}
 \alpha^x \leq \liminf_{G\rightarrow \infty}\frac{t(G+x)}{t(G)} \leq \limsup_{G\rightarrow \infty} \frac{t(G+x)}{t(G)}.
 \label{eq:nvar-proof-bounds}
\end{equation}
A possible value is $\alpha=1$.
Similar to the proof of Theorem \ref{th:onevar-cvg}, we use the notation
\begin{align*}
 \underline{Z} = \liminf_{G\rightarrow \infty}\frac{t(G+z)}{t(G)} && \bar{Z} = \limsup_{G\rightarrow \infty}\frac{t(G+z)}{t(G)}.
\end{align*}
For all variables $i$, we have:
\begin{align*}
 \bar{Z}
 &= \limsup_{G\rightarrow \infty} \left( \frac{t(G+z)}{t(G+z-l_i)} \frac{t(G+z-l_i)}{t(G+z-2 l_i)} \dots \frac{t(G+ z -(\frac{z}{l_i}-1)l_i)}{t(G)} \right)\\
 &\leq \left( \limsup_{G\rightarrow \infty} \frac{t(G+l_i)}{t(G)}  \right)^\frac{z}{l_i}\\
 &= \left( \limsup_{G\rightarrow \infty} \frac{1 + \min_j(t(G+l_i-l_j) + t(G+l_i - r_j))}{t(G)}  \right)^\frac{z}{l_i}\\
 &\leq \left( 1 + \limsup_{G\rightarrow \infty} \frac{t(G+l_i - r_i)}{t(G)}  \right)^\frac{z}{l_i}\\
 &\leq \left( 1 + \left(\liminf_{G\rightarrow \infty} \frac{t(G)}{t(G+l_i - r_i)}\right)^{-1}  \right)^\frac{z}{l_i}\\
 &\leq \left( 1 + \left(\liminf_{G\rightarrow \infty} \frac{t(G+r_i - l_i)}{t(G)}\right)^{-1}  \right)^\frac{z}{l_i}\\
 &\leq \left( 1 + \alpha^{l_i-r_i} \right)^\frac{z}{l_i}
\end{align*}
where the last line follows from the lower bound on $\alpha^{-x}$ from \eqref{eq:nvar-proof-bounds}.

Suppose there exists a variable $i$ such that $\alpha > \varphi_i$, then we establish, using Theorem \ref{th:pol}:
\begin{align*}
 \alpha > \varphi_i
 &\Rightarrow p_i(\alpha) > 0\\
 &\Rightarrow \alpha^{r_i} - \alpha^{r_i-l_i} - 1 > 0\\
 &\Rightarrow 1 + \alpha^{l_i-r_i} < \alpha^{l_i}\\
 &\Rightarrow (1 + \alpha^{l_i-r_i})^\frac{z}{l_i} < (\alpha^{l_i})^\frac{z}{l_i}\\
 &\Rightarrow \bar{Z} < \alpha^z.
\end{align*}
This contradicts expression \eqref{eq:nvar-proof-bounds}, hence for all variables $i$, $\alpha \leq \varphi_i$.
Suppose that there exists a variable $i$ such that $\varphi_i^z < \bar{Z}$, then
\begin{align*}
 \alpha^z \leq \varphi_i^z < \bar{Z} \leq \left( 1 + \alpha^{l_i-r_i} \right)^\frac{z}{l_i}
 & \Rightarrow
 \alpha^{l_i} < 1 + \alpha^{l_i-r_i}\\
 & \Rightarrow
 \alpha^{r_i-l_i} p_i(\alpha) < 0\\
 & \Rightarrow
 \alpha < \varphi_i,
\end{align*}
which is a contradiction, hence for all variables $i$, $\bar{Z} \leq \varphi_i^z$.
In addition, for each variable $i$,
\begin{align*}
 \underline{Z}
 &\geq \left( \liminf_{G\rightarrow \infty} \frac{t(G+r_i)}{t(G)}  \right)^\frac{z}{r_i}\\
 &\geq \left( \liminf_{G\rightarrow \infty} \frac{1 + \min_j(t(G+r_i-l_j) + t(G+r_i-r_j))}{t(G)}  \right)^\frac{z}{r_i} \\
 &\geq \left( \min_j \liminf_{G\rightarrow \infty} \frac{t(G+r_i-l_j) + t(G+r_i-r_j)}{t(G)}  \right)^\frac{z}{r_i} .\\
 \intertext{Hence there must exist a variable $j$ such that}
 \underline{Z}
 &\geq \left( 1 + \liminf_{G\rightarrow \infty} \frac{t(G+r_j - l_j)}{t(G)}  \right)^\frac{z}{r_j}\\
 &\geq \left( 1 + \alpha^{r_j-l_j} \right)^\frac{z}{r_j}.
\end{align*}
Suppose $\alpha < \varphi_j$, then, using Theorem \ref{th:pol} in the first line:
\begin{align*}
\alpha < \varphi_j &\Rightarrow p_j(\alpha) < 0\\
&\Rightarrow 1 + \alpha^{r_j-l_j} > \alpha^{r_j} \\
&\Rightarrow (1 + \alpha^{r_j-l_j})^\frac{z}{r_j} > (\alpha^{r_j})^\frac{z}{r_j} \\
&\Rightarrow \underline{Z} > \alpha^z.
\end{align*}
Since $\varphi_j^z > \alpha^z$ implies $\underline{Z} > \alpha^z$, then $\underline{Z} \geq \varphi_j^z$ must be true.
This can be shortly proven by writing
\begin{align*}
 (\varphi_j^z > \alpha^z \Rightarrow \underline{Z} > \alpha^z) \Rightarrow \neg(\varphi_j^z > \alpha^z \land \underline{Z} \leq \alpha^z) \Rightarrow \neg(\varphi_j^z > \underline{Z}) \Rightarrow \varphi_j^z \leq \underline{Z}.
\end{align*}
Since $\varphi_i^z \geq \bar{Z}$ holds for all variables $i$, then variable $j$ satisfies $\varphi_j^z \leq \underline{Z} \leq \bar{Z} \leq \varphi_i^z$ for all variables $i$.
We can finally conclude that
\begin{align*}
 \lim_{G \rightarrow \infty} \sqrt[z]{\frac{t(G+z)}{t(G)}} = \min_{i} \varphi_i^z.
\end{align*}

\end{proof}

\section{Additional numerical simulations}\label{app-sec:additional_simulations}
We give an additional set of simulations on the \nvar{} model, with the same target gaps as in the simulations on \gap{}, so that one can compare both experimental setups easily.
For the same experiment, we present two different tables, Table \ref{tab:simulation-nvar-more-gaps-presented-as-nvar} and \ref{tab:simulation-nvar-more-gaps-presented-as-gap}, where the results are presented as in Table \ref{tab:simulation-nvar} and \ref{tab:simulation-gap}, respectively.
The two differences between Table \ref{tab:simulation-nvar-more-gaps-presented-as-nvar} and \ref{tab:simulation-nvar-more-gaps-presented-as-gap} are thus the presence (or absence) of the last column LB, and the reference used for relative performance (for Table \ref{tab:simulation-nvar-more-gaps-presented-as-nvar} it is the minimum tree-size, and for Table \ref{tab:simulation-nvar-more-gaps-presented-as-gap} it is the tree-size produced by \sfproduct{}).

These results show that \sfratio{} performs generally better than \sfproduct{} and \sflinear{}, and that this phenomenon becomes more significant as the gap to close increases.
If we compare Table \ref{tab:simulation-nvar-more-gaps-presented-as-gap} to Table \ref{tab:simulation-gap}, it appears that \sfratio{} is even more at an advantage on \gap{} than on \nvar{}.
One reason may be that, for a given instance, the best variable for \sfratio{} and \sfproduct{} may be the same, or not very different, and in the \nvar{} experiments this best variable would be branched on at every node.
However, in \gap{}, this variable would be branched on only at the root node, and the subsequent best variables chosen by \sfratio{} and \sfproduct{} are likely to differ.
Indeed, observe that there are many more ties in the \nvar{} experiments than in the \gap{} experiments.

\begin{table}
\centering
\begin{footnotesize}
\begin{tabular}{c|c|c|ccccc|c|c|c}
\multirow{2}{*}{Data}& \multirow{2}{*}{Gap} & & \multicolumn{5}{c|}{\sflinear{}($\mu$)} & \multirow{2}{*}{\sfproduct{}} & \multirow{2}{*}{\sfratio{}} & \multirow{2}{*}{LB}\\
&& $\mu =$ & $0$ & $\frac{1}{6}$ & $\frac{1}{3}$ &  $\frac{1}{2}$ & $1$ & &  \\
\hline
\hline
\multirow{6}{*}{B}&\multirow{2}{*}{5000}& t-s & \textbf{0.13} & 0.15 & 0.15 & 0.46 & 517.40 & 0.16 & 0.16 & 0.11\\
&& wins & \textbf{99} & 98 & 98 & 96 & 21 & 97 & 97 & \\
\cline{2-11}
&\multirow{2}{*}{15000}& t-s & 6.02 & 3.49 & 3.13 & 2.80 & $10^{04}$ & \textbf{2.76} & \textbf{2.76} & 2.62\\
&& wins & 89 & 94 & 96 & 98 & 8 & \textbf{99} & \textbf{99} & \\
\cline{2-11}
&\multirow{2}{*}{25000}& t-s & 7.78 & 3.68 & 3.18 & 2.72 & $10^{05}$ & \textbf{2.71} & \textbf{2.71} & 2.71\\
&& wins & 88 & 93 & 95 & 99 & 8 & \textbf{100} & \textbf{100} & \\
\hline
\hline
\multirow{6}{*}{U}&\multirow{2}{*}{5000}& t-s & 98.02 & 2.47 & 0.96 & 0.90 & 1353.05 & 0.96 & \textbf{0.87} & 0.73\\
&& wins & 16 & 85 & 95 & 89 & 12 & \textbf{96} & 93 & \\
\cline{2-11}
&\multirow{2}{*}{12500}& t-s & 544.24 & 10.22 & \textbf{2.76} & 3.59 & $10^{04}$ & 2.82 & 2.79 & 2.50\\
&& wins & 13 & 74 & \textbf{99} & 84 & 10 & 97 & 92 & \\
\cline{2-11}
&\multirow{2}{*}{20000}& t-s & 1860.53 & 16.64 & 3.58 & 4.15 & $10^{06}$ & 3.80 & \textbf{3.17} & 3.09\\
&& wins & 13 & 70 & 92 & 87 & 11 & 90 & \textbf{95} & \\
\hline
\hline
\multirow{6}{*}{V} &  \multirow{2}{*}{5000}& t-s & 694.14 & 7.43 & 7.49 & 15.47 & 1587.97 & 6.52 & \textbf{5.80} & 5.49\\
&& wins & 20 & 82 & 81 & 61 & 13 & 87 & \textbf{92} & \\
\cline{2-11}
&  \multirow{2}{*}{7500}& t-s & 2176.92 & 7.79 & 7.81 & 18.20 & 5596.92 & 6.75 & \textbf{5.89} & 5.69\\
&& wins & 19 & 81 & 83 & 64 & 14 & 86 & \textbf{94} & \\
\cline{2-11}
&\multirow{2}{*}{10000}& t-s & 6456.22 & 9.48 & 8.17 & 22.52 & $10^{04}$ & 7.61 & \textbf{5.81} & 5.72\\
&& wins & 19 & 76 & 86 & 66 & 15 & 83 & \textbf{97} & \\
\hline
\hline
\multirow{6}{*}{X} & \multirow{2}{*}{2000}& t-s & 626.55 & 5.64 & 14.14 & 36.89 & 527.77 & \textbf{5.08} & 5.64 & 4.79\\
&& wins & 11 & 85 & 64 & 45 & 12 & \textbf{91} & 85 & \\
\cline{2-11}
& \multirow{2}{*}{3000}& t-s & 1943.09 & \textbf{6.21} & 16.60 & 48.57 & 1198.56 & 6.76 & \textbf{6.21} & 5.67\\
&& wins & 9 & \textbf{88} & 68 & 47 & 12 & \textbf{88} & \textbf{88} & \\
\cline{2-11}
& \multirow{2}{*}{4000}& t-s & 5797.74 & \textbf{6.03} & 18.73 & 60.70 & 2449.16 & 7.30 & \textbf{6.03} & 5.68\\
&& wins & 9 & \textbf{93} & 67 & 47 & 12 & 89 & \textbf{93} & \\
\end{tabular}
\end{footnotesize}
\caption{Simulation results on \nvar{} presented as Table \ref{tab:simulation-nvar}.}
 \label{tab:simulation-nvar-more-gaps-presented-as-nvar}
\end{table}

\begin{table}
\centering
\begin{footnotesize}
\begin{tabular}{c|c|c|ccccc|c|c}
\multirow{2}{*}{Data}& \multirow{2}{*}{Gap} & & \multicolumn{5}{c|}{\sflinear{}($\mu$)} & \multirow{2}{*}{\sfproduct{}} & \multirow{2}{*}{\sfratio{}}\\
&& $\mu =$ & $0$ & $\frac{1}{6}$ & $\frac{1}{3}$ &  $\frac{1}{2}$ & $1$ & &  \\
\hline
\hline
\multirow{6}{*}{B}&\multirow{2}{*}{5000}& t-s & \textbf{-0.03} & -0.02 & -0.02 & 0.30 & 516.40 & 0.00 & 0.00\\
&& wins & \textbf{99} & 98 & 98 & 96 & 21 & 97 & 97\\
\cline{2-10}
&\multirow{2}{*}{15000}& t-s & 3.18 & 0.71 & 0.36 & 0.04 & $10^{04}$ & \textbf{0.00} & \textbf{0.00}\\
&& wins & 89 & 94 & 96 & 98 & 8 & \textbf{99} & \textbf{99}\\
\cline{2-10}
&\multirow{2}{*}{25000}& t-s & 4.94 & 0.94 & 0.46 & 0.01 & $10^{05}$ & \textbf{0.00} & \textbf{0.00}\\
&& wins & 88 & 93 & 95 & 99 & 8 & \textbf{100} & \textbf{100}\\
\hline
\hline
\multirow{6}{*}{U}&\multirow{2}{*}{5000}& t-s & 96.14 & 1.50 & 0.00 & -0.06 & 1339.26 & 0.00 & \textbf{-0.09}\\
&& wins & 16 & 85 & 95 & 89 & 12 & \textbf{96} & 93\\
\cline{2-10}
&\multirow{2}{*}{12500}& t-s & 526.58 & 7.20 & \textbf{-0.06} & 0.75 & $10^{04}$ & 0.00 & -0.03\\
&& wins & 13 & 74 & \textbf{99} & 84 & 10 & 97 & 92\\
\cline{2-10}
&\multirow{2}{*}{20000}& t-s & 1788.75 & 12.37 & -0.21 & 0.34 & $10^{06}$ & 0.00 & \textbf{-0.61}\\
&& wins & 13 & 70 & 92 & 87 & 11 & 90 & \textbf{95}\\
\hline
\hline
\multirow{6}{*}{V} &  \multirow{2}{*}{5000}& t-s & 645.51 & 0.85 & 0.90 & 8.40 & 1484.61 & 0.00 & \textbf{-0.68}\\
&& wins & 20 & 82 & 81 & 61 & 13 & 87 & \textbf{92}\\
\cline{2-10}
&  \multirow{2}{*}{7500}& t-s & 2032.98 & 0.97 & 0.99 & 10.73 & 5236.79 & 0.00 & \textbf{-0.81}\\
&& wins & 19 & 81 & 83 & 64 & 14 & 86 & \textbf{94}\\
\cline{2-10}
&\multirow{2}{*}{10000}& t-s & 5992.84 & 1.74 & 0.52 & 13.86 & $10^{04}$ & 0.00 & \textbf{-1.67}\\
&& wins & 19 & 76 & 86 & 66 & 15 & 83 & \textbf{97}\\
\hline
\hline
\multirow{6}{*}{X} & \multirow{2}{*}{2000}& t-s & 591.45 & 0.54 & 8.62 & 30.28 & 497.44 & \textbf{0.00} & 0.54\\
&& wins & 11 & 85 & 64 & 45 & 12 & \textbf{91} & 85\\
\cline{2-10}
& \multirow{2}{*}{3000}& t-s & 1813.72 & \textbf{-0.52} & 9.21 & 39.16 & 1116.33 & 0.00 & \textbf{-0.52}\\
&& wins & 9 & \textbf{88} & 68 & 47 & 12 & \textbf{88} & \textbf{88}\\
\cline{2-10}
& \multirow{2}{*}{4000}& t-s & 5396.50 & \textbf{-1.18} & 10.65 & 49.77 & 2275.74 & 0.00 & \textbf{-1.18}\\
&& wins & 9 & \textbf{93} & 67 & 47 & 12 & 89 & \textbf{93}\\
\end{tabular}
\end{footnotesize}
\caption{Simulation results on \nvar{} presented as Table \ref{tab:simulation-gap}.}
 \label{tab:simulation-nvar-more-gaps-presented-as-gap}
\end{table}

\OS{}

\end{document}